\documentclass[11pt]{article}

\usepackage{amsmath, amsfonts, amssymb, amsthm,  graphicx, mathtools, enumerate}
\usepackage{subcaption,caption}

\usepackage[blocks, affil-it]{authblk}

\usepackage[numbers, square]{natbib}
\usepackage[CJKbookmarks=true,
            bookmarksnumbered=true,
			bookmarksopen=true,
			colorlinks=true,
			citecolor=red,
			linkcolor=blue,
			anchorcolor=red,
			urlcolor=blue]{hyperref}
\usepackage[usenames]{color}

\usepackage[letterpaper, left=1.2truein, right=1.2truein, top = 1.2truein, bottom = 1.2truein]{geometry}
\usepackage[ruled, vlined, lined, commentsnumbered]{algorithm2e}

\usepackage{prettyref,soul}

\newtheorem{lemma}{Lemma}[section]

\newtheorem{thm}{Theorem}[section]

\newtheorem{corollary}{Corollary}[section]

\newrefformat{eq}{(\ref{#1})}
\newrefformat{chap}{Chapter~\ref{#1}}
\newrefformat{sec}{Section~\ref{#1}}
\newrefformat{algo}{Algorithm~\ref{#1}}
\newrefformat{fig}{Fig.~\ref{#1}}
\newrefformat{tab}{Table~\ref{#1}}
\newrefformat{rmk}{Remark~\ref{#1}}
\newrefformat{clm}{Claim~\ref{#1}}
\newrefformat{def}{Definition~\ref{#1}}
\newrefformat{cor}{Corollary~\ref{#1}}
\newrefformat{lmm}{Lemma~\ref{#1}}
\newrefformat{lemma}{Lemma~\ref{#1}}
\newrefformat{prop}{Proposition~\ref{#1}}
\newrefformat{app}{Appendix~\ref{#1}}
\newrefformat{ex}{Example~\ref{#1}}
\newrefformat{exer}{Exercise~\ref{#1}}
\newrefformat{soln}{Solution~\ref{#1}}
\newrefformat{cond}{Condition~\ref{#1}}

\def\text#1{\mbox{\rm #1}}

\newcommand{\argmin}{\mathop{\rm argmin}}
\newcommand{\argmax}{\mathop{\rm argmax}}

\newcommand{\E}{\mathbb{E}}
\newcommand{\Prob}{\mathbb{P}}

\newcommand{\RN}[1]{%
  \textup{\uppercase\expandafter{\romannumeral#1}}%
}

\newcommand{\thetas}{\hat{\theta}^{(s)}}

\newcommand{\zs}{\hat{z}^{(s)}}
\newcommand{\zsnext}{\hat{z}^{(s+1)}}

\newcommand{\iprod}[2]{\left\langle #1, #2 \right\rangle}

\title{Statistical and Computational Guarantees of Lloyd's Algorithm and Its Variants}
\author[1]{Yu Lu}
\author[1]{Harrison H. Zhou}
\affil[1]{Yale University}

\begin{document}
\maketitle

\begin{abstract}
Clustering is a fundamental problem in statistics and machine learning. Lloyd's algorithm, proposed in 1957, is still possibly the most widely used clustering algorithm in practice due to its simplicity and empirical performance. However, there has been little theoretical investigation on the statistical and computational guarantees of Lloyd's algorithm. This paper is an attempt to bridge this gap between practice and theory. We investigate the performance of Lloyd's algorithm on clustering sub-Gaussian mixtures. Under an appropriate initialization for labels or centers, we show that Lloyd's algorithm converges to an exponentially small clustering error after an order of $\log n$ iterations, where $n$ is the sample size. The error rate is shown to be minimax optimal.  For the two-mixture case, we only require the initializer to be slightly better than random guess.  

In addition, we extend the Lloyd's algorithm and its analysis to community detection and crowdsourcing, two problems that have received a lot of attention recently in statistics and machine learning. Two variants of Lloyd's algorithm are proposed respectively for community detection and crowdsourcing. On the theoretical side, we provide statistical and computational guarantees of the two algorithms, and the results improve upon some previous signal-to-noise ratio conditions in literature for both problems. Experimental results on simulated and real data sets demonstrate competitive performance of our algorithms to the state-of-the-art methods.
\end{abstract}


\section{Introduction}
Lloyd's algorithm, proposed in 1957 by Stuart Lloyd at Bell Labs \cite{lloyd1982least}, is still one of the most popular clustering algorithms used by practitioners, with a wide range of applications from computer vision \cite{agarwal2004k}, to astronomy \cite{ordovas2014fast} and to biology \cite{herwig1999large}. Although considerable innovations have been made on developing new provable and efficient clustering algorithms 
in the past six decades, Lloyd's algorithm has been consistently listed as one of the top ten data mining algorithms in several recent surveys \cite{wu2008top}. 

Lloyd's algorithm is very simple and easy to implement. It starts with an initial estimate of centers or labels and then iteratively updates the labels and the centers until convergence. Despite its simplicity and a wide range of successful applications, surprisingly, there is little theoretical analysis on explaining the effectiveness of Lloyd's algorithm. It is well known that there are two issues with Lloyd's algorithm under the worst case analysis. First, as a greedy algorithm, Lloyd's algorithm is only guaranteed to converge to a local minimum \cite{milligan1980examination}.
The $k$-means objective function that Lloyd's algorithm attempts to minimize is NP-hard \cite{dasgupta2008hardness, mahajan2009planar}.   
Second, the convergence rate of Lloyd's algorithm can be very slow. Arthur and Vassilvitskii \cite{arthur2006slow} construct a worst-case showing that Lloyd's algorithm can require a superpolynomial running time. 

A main goal of this paper is trying to bridge this gap between theory and practice of Lloyd's algorithm. We analyze its performance on the Gaussian mixture model \cite{pearson1894contributions, titterington1985statistical}, a standard model for clustering, and consider the generalization to sub-Gaussian mixtures, which includes binary observations as a special case. Specifically, we attempt to address following questions to help understand Lloyd's algorithm: How good does the initializer need to be? How fast does the algorithm converge? What separation conditions do we need? What is the clustering error rate and how it is compared with the optimal statistical accuracy?


Despite the popularity of Lloyd's algorithm as a standard procedure for the k-means problem, to the best of our knowledge,
there is little work in statistics to understand the algorithm. Some efforts have been made by computer scientists to develop effective initialization techniques for Lloyd's algorithm \cite{arthur2007k, ostrovsky2006effectiveness, aggarwal2009adaptive}. Their main focus is to find polynomial-time approximation scheme of the k-means objective function rather than to identify the cluster labels of data points, which is often the primary interest for many applications. It is worthwhile to emphasize that when the signal-to-noise ratio is not sufficiently large, a small k-means objective function value does not necessarily guarantee a small clustering error. Furthermore, the error rate is different from the optimal error of an exponential form. Recently, some fascinating results by Kumar and Kannan \cite{kumar2010clustering} and  Awasthi and Sheffet  \cite{awasthi2012improved} show that under certain strong separation conditions, the Lloyd's algorithm initialized by spectral clustering correctly classifies all data points with high probability. However, they focus on the strong consistency results and the clustering error rate of the Lloyd's algorithm remains unclear.  It is desirable to have a systematic study of Lloyd's algorithm under various separation conditions such that the strong consistency can be included as a special case.

Lloyd's algorithm is an iterative procedure. Its analysis can be challenging due to dependence between iterative steps. In the statistics literature, various two-stage estimators, or more precisely, two-step estimators, have been proposed to successfully solve some very important non-convex problems, for example, sparse principle analysis \cite{cai2013sparse, wang2014nonconvex}, community detection \cite{gao2015achieving, gao2016community}, mixture of linear regression \cite{chaganty2013spectral, yi2014alternating} and crowdsourcing \cite{zhang2014spectral, gao2016exact}. For all those problems, under a strong assumption that the initial estimator is consistent, one-step update in the second stage usually leads us to a minimax optimal estimator. However, as observed in various simulation or real data studies in \cite{zhang2014spectral} and \cite{gao2015achieving}, the initial estimator may perform poorly, and more iterations in the second stage keeps driving down the clustering error.  Unfortunately, due to some technical difficulties, theoretical analyses in \cite{zhang2014spectral} and \cite{gao2015achieving} restrict to one-step iteration. 


\subsection{Our contributions} In this paper, we give a considerably weak initialization condition under which Lloyd's algorithm converges to the optimal label estimators of sub-Gaussian mixture model. While previous results \cite{kumar2010clustering, awasthi2012improved} focus on exact recovery (strong consistency) of the labels, we obtain the clustering error rates of Lloyd's algorithm under various signal-to-noise levels. As a special case, we obtain exact recovery with high probability when the signal-to-noise level is bigger than $4\log n$. The signal-to-noise ratio condition for exact recovery is weaker than the state-of-the-art result \cite{awasthi2012improved}. In contrast to previous two-stage (two-step) estimators, our analyses go beyond one-step update. We are able to show a linear convergence to the statistical optimal error rate for Lloyd's algorithms and its two variants for community detection and crowdsourcing. 

We illustrate our contributions here by considering  the problem of clustering two-component spherical Gaussian mixtures, with symmetric centers $\theta^*$ and $-\theta^* \in \mathbb{R}^d$ and variance $\sigma^2$. See Section \ref{sec:prelim} for more details. Let $n$ be the sample size and $r=\|\theta^*\|/(\sigma \sqrt{1+9d/n})$ be the normalized signal-to-noise ratio. We establish the following basin of attractions of Lloyd's algorithm. 

\begin{thm} \label{thm:inf1} Assume $r \ge C$ and $n \ge C$ for a sufficiently large constant $C$. For symmetric, two-component spherical Gaussian mixtures,  given any initial estimator of labels with clustering error $$A_0 < \frac{1}{2} - \frac{2.56+\sqrt{\log r}}{r}- \frac{1}{\sqrt{n}}, \quad w.h.p.$$ Lloyd's algorithm converges linearly to an exponentially small rate after $\lceil 3\log n\rceil$ iterations, which is the minimax rate as $r  \to \infty$ w.h.p.
\end{thm}

The results above are extended to general number of clusters $k$ and to (non-spherical) sub-Gaussian distributions under an appropriate initialization condition and a signal-to-noise ratio condition, which, to the best of our knowledge, are the weakest conditions in literature. 

Now we discuss the contributions of this paper in detail. Our contributions are three folds. First, we give statistical guarantees of Lloyd's algorithm. Starting with constant clustering error, by an appropriate initializer, such as spectral clustering, we show an exponentially small clustering error rate of Lloyd's algorithms, under a weak signal-to-noise ratio condition. We also provide a rate-matching lower bound to show that Lloyd's algorithm initialized by spectral clustering is minimax optimal. When the clusters sizes are of the same order and the distance between different centers are of the same order, our signal-to-noise condition reduces to $\Delta \gtrsim \sigma k \sqrt{1+kd/n}$, where $\Delta$ is the minimum Euclidean distances between two different cluster centers. Previous results on Lloyd's algorithms focus on exact recovery of the labels and they assume sample size $n \gg kd$ \cite{kumar2010clustering, awasthi2012improved}. The best known signal-to-noise ratio condition is $\Delta \gtrsim \sigma k \textrm{ polylog } n$ \cite{awasthi2012improved}. Our condition is weaker by a $\textrm{polylog }n$ factor. Moreover, our results also hold for the high-dimensional case where $d$ could be larger than $n$.  

Second, we provide computational guarantees of Lloyd's algorithm. We prove a linear convergence rate of Lloyd's iterations given a label initializer of constant clustering error. The linear convergence rate depends on the signal-to-noise ratio. Larger signal-to-noise ratio leads to faster convergence rate. Counterexamples are constructed in the appendix to show that the basin of attractions we establish in Theorem \ref{thm:inf1} and in the general $k$ case are almost necessary.  It is worthwhile to point out that the initialization condition in Theorem \ref{thm:inf1} is just slightly stronger than random guess. As implied by the minimax lower bound (Theorem \ref{thm:lower_bound}), a necessary condition for consistently estimating the labels is $\frac{\Delta}{\sigma} \to \infty$. In addition, when $\Delta \gtrsim \sigma d^{1/4}$, we are able to prove a random initialization scheme works with high probability. 

Third, we develop new proof techniques for analyzing two-stage estimators. 
Previous analyses usually require the first stage's estimator to be consistent to achieve the minimax optimal rates \cite{zhang2014spectral, gao2015achieving}. In contrast, we only need a constant clustering error of the first stage estimator by introducing a new technique to analyze random indicator functions. In addition, we are able to improve previous signal-to-noise ratio conditions in community detection and crowdsourcing by considering two variants of Lloyd's algorithms. Simulated and real data experiments show our algorithms are competitive as compared to the state-of-the-art algorithms for both problems. 

\subsection{Related work on Gaussian mixture models} \label{sec:related} The study of Gaussian mixture model \cite{pearson1894contributions} has a long and rich history. We give a very brief review here. In the original paper of Pearson \cite{pearson1894contributions}, methods of moments estimators were first proposed, followed by the work of \cite{day1969estimating, lindsay1993multivariate}. They all involve solving high-order moments equations, which is computationally challenging as the sample size or the dimensionality grows. Recently, third-order moments was proved to be sufficient by using tensor decomposition techniques \cite{chang1996full, anandkumar2012method, hsu2013learning}. However, their sample size requirement is in high-order polynomial of $d$ and $k$ thus can not be extended to the high-dimensional case. Another line of research focuses on spectral projections and their variants. By using different dimension reduction technique, they keep improving the separation condition for general Gaussian mixture models \cite{mcsherry2001spectral, achlioptas2005spectral, kannan2005spectral, moitra2010settling, hardt2015tight} and spherical Gaussian mixture models \cite{vempala2004spectral}. These works almost all focus on estimating the centers.  Little is known about the convergence rate of estimating the cluster labels. 

Another popular algorithm for Gaussian mixture model is the EM algorithm \cite{dempster1977maximum}. \cite{xu1996convergence} analyzed the local convergence of EM for well-separated Gaussian mixtures. \cite{dasgupta2007probabilistic} showed that a two-round variant of EM algorithm with a special initializer is consistent under a strong separation condition that $\Delta \gtrsim \sigma \sqrt{d}$. Recently, a lot of attention has been gained on the computational guarantees of EM algorithm. Two component, symmetric Gaussian mixture is most widely studied due to its simple structure. For this model, \cite{balakrishnan2014statistical} first proves the linear convergence of EM if the center is initialized in a small neighborhood of the true parameters, which implies an exponentially small error rate of labels after one-step label update. \cite{klusowski2016statistical} extends their basin of attraction to be the intersection of a half space and a ball near the origin. More recently, \cite{xu2016global, daskalakis2016ten} prove the global convergence of the EM algorithm given infinite samples. While these results are encouraging, it is unclear whether their technique can be generalized to the $k$-mixtures or non-Gaussian cases. \cite{jin2016local} provides some interesting negative results of EM algorithms. When $k \ge 3$, counterexamples are constructed to show there is no general global convergence of EM algorithm by uniformly initializing the centers from data points.

\subsection{Organization and Notation}
The rest of this paper is organized as follows. We introduce the sub-Gaussian mixture problem and Lloyd's algorithm in Section \ref{sec:prelim}. Section \ref{sec:main} gives statistical and computational guarantees of Lloyd's algorithm, and discusses their implications when using spectral initialization. In Section \ref{sec:app}, we consider two variants of Lloyd's algorithms for community detection and crowdsourcing and establish their theoretical properties. A number of experiments are conducted in Section \ref{sec:data} that confirm the sharpness of our theoretical findings and demonstrate the effectiveness of our algorithms on real data.  In Section \ref{sec:diss}, we discuss our results on random initialization and the error of estimating centers. Finally, we present the main proofs in Section \ref{sec:proof}, with more technical part of the proofs deferred to the appendix. 

Throughout the paper, let $ [m] = \{1,2,\cdots,m\}$ for any positive integer $m$. For any vector $a$, $\|a\|$ is the $\ell_2$ norm of $a$. $\mathbb{I}\{ \cdot \}$ is the indicator function. Given two positive sequences $\{a_n\}$ and $\{b_n\}$, $a_n \gtrsim b_n$ means there is a universal constant $C$ such that $a_n \ge C b_n$ for all $n$, and define $a_n \lesssim b_n$ vice versa. We write $a_n \asymp b_n$ if $a_n \gtrsim b_n$ and $a_n \lesssim b_n$.  Denote $a_n = o(b_n)$ if $a_n/b_n \to 0$ as $n \to \infty.$

\section{Model and Lloyd's algorithm} \label{sec:prelim}
In this section, we first introduce the mixture of sub-Gaussians, a standard model for $k$-means, then briefly review Lloyd's algorithm, followed by introducing spectral clustering algorithms as a way of initialization.  
\subsection{Mixture of sub-Gaussians}
Suppose we observe independent samples $y_1, y_2, \cdots, y_n \in \mathbb{R}^d$ from a mixture of $k$ sub-Gaussian distributions, \begin{equation} \label{eq:model} 
y_i = \theta_{z_i} + w_i  \;\textrm{ for } i \in [n],
\end{equation}
where $z_1, z_2, \cdots, z_n \in [k]$ are the underlying labels,  and $\theta_1, \cdots, \theta_k \in \mathbb{R}^d$ are unknown centers of those $k$ distributions. 
We assume the noise $\{w_{i}, i \in [n]\}$ are independent zero mean sub-Gaussian vectors with parameter $\sigma>0$, i.e.
\begin{equation} \label{eq:subgaussian} 
\E e^{\iprod{a}{w_{i}}} \le e^{\frac{\sigma^2\|a\|^2}{2}}, \quad \textrm{for all } i \in [n] \textrm{ and } a \in \mathbb{R}^{d}.
\end{equation} 

A special case of model (\ref{eq:model}) is the symmetric, two-component mixture model, in which the two centers are $\theta^*$ and $-\theta^*$. We observe $y_1, y_2, \cdots, y_n$ from the following generative model,
\begin{equation} \label{eq:modeltwo}
y_i = z_i \theta^* + \xi_i,
\end{equation}
where $z_i \in \{-1,1\}$ and $\{\xi_{i}, i \in [n]\}$ are independent Gaussian noise with covariance matrix $\sigma^2 I_d$. Here with a little abuse of notation, we also use $z_i \in \{-1,1\}$ to denote the underlying labels. As arguably the simplest mixture model, this special model has recently gained some attention in studying the convergence of EM algorithm \cite{balakrishnan2014statistical, klusowski2016statistical, xu2016global}.  Other examples of model (\ref{eq:model}) are planted partition model \cite{mcsherry2001spectral} for random graphs, stochastic block model \cite{holland1983stochastic} for network data analysis and Dawid-Skene model \cite{dawid1979maximum} for crowdsourcing (see Section \ref{sec:app} for more tails). 

For these clustering problems, our main goal is to recover the unknown labels $z_i$ rather than to estimate the centers $\{\theta_j\}$. Note that the cluster structure is invariant to the permutations of label symbols. We define the mis-clustering rate of estimated labels $\hat{z}_1, \cdots, \hat{z}_n$ as
\begin{equation} \label{eq:loss}
L(\hat{z}, z) = \inf_{\pi \in \mathcal{S}_k} \left[ \frac{1}{n} \sum_{i=1}^{n} \mathbb{I}\left\{ \pi(\hat{z}_i) \neq z_i \right\} \right], 
\end{equation}
where $\mathcal{S}_k$ is the collection of all the mappings from $[k]$ to $[k]$.

\subsection{Lloyd's Algorithm}
Lloyd's algorithm was originally proposed to solve the following $k$-means problem. Given $n$ vectors $y_1, y_2, \cdots, y_n \in \mathbb{R}^d$ and an integer $k$, the goal is to find $k$ points $\theta_1, \theta_2, \cdots,  \theta_k \in \mathbb{R}^d$ to minimize the following objective function
\begin{equation} \label{eq:kmeans}
\sum_{i \in [n]} \min_{j \in [k]} \|y_i - \theta_j\|^2.
\end{equation}
The problem above is equivalent to find $\hat{\theta}_1, \cdots, \hat{\theta}_k \in \mathbb{R}^d$ and $\hat{z}_1, \cdots, \hat{z}_n \in [k]$ such that
\begin{equation} \label{eq:kmeans2}
(\hat{\theta}, \hat{z}) = \argmin_{(\theta, z)} \sum_{i \in [n]} \left\|y_i - \sum_{j=1}^{k} \theta_j \mathbb{I}\{z_i = j\} \right\|^2.
\end{equation}
From a statistical point of view, (\ref{eq:kmeans2}) is the maximum likelihood estimator of model (\ref{eq:model}) with spherical Gaussian noise. It has been proved by Pollard \cite{pollard1981strong, pollard1982central} the strong consistency and central limit theorem of using (\ref{eq:kmeans}) to estimate the centers $\theta$. 

Lloyd's algorithm is simply motivated by the following observation of (\ref{eq:kmeans2}). If we fix $\theta$, $\hat{z}_i$ is the index of the center that $y_i$ is closest to, and if we fix $z$, $\hat{\theta}_j$ is the sample mean of those $y_i$ with $z_i = j$. If we start with an initial estimate of centers or labels, we can iteratively update the labels and centers. Therefore, we have the following Lloyd's algorithm.
\begin{itemize}
\item [1.] Get an initial estimate of the centers or the labels.
\item [2.] Repeat the following iteration until convergence.
\begin{itemize}
 \item[2a.] \text{For} {$h=1,2,\cdots k$}, 
\begin{equation} \label{eq:centerupdate} 
\thetas_h = \frac{\sum_{i=1}^{n} y_i \mathbb{I}\{\hat{z}^{(s)}_i=h\}}{ \sum_{i=1}^{n} \mathbb{I}\{\hat{z}^{(s)}_i=h\}}.
\end{equation}
\item[2b.]  \text{For} {$i=1,2,\cdots n$}, 
\begin{equation} \label{eq:labelupdate} 
\zsnext_i =  \argmin_{h \in [k]} \|y_i - \thetas_h \|^2.
\end{equation}
\end{itemize}
 \end{itemize}
For model (\ref{eq:modeltwo}), since the centers are parametrized by a single vector $\theta^*$, the center update step (\ref{eq:centerupdate}) can be simplified to 
\begin{eqnarray}
\label{eq:center} \thetas &=& \frac{1}{n} \sum_{i=1}^{n} \hat{z}^{(s)}_i y_i.
\end{eqnarray}

Lloyd's algorithm is guaranteed to converge because it is a greedy algorithm that keeps decreasing the value of objective function (\ref{eq:kmeans2}). However, it may only converge to a local minimum and thus a good initializer is needed. 
%
%
%
One possible way of initializing Lloyd's algorithm is to use spectral methods. Note that the signal matrix $\theta$ in model (\ref{eq:model}) is low rank. Spectral methods first do a de-noising step by projecting the data onto the subspace spanned by its top singular vectors, which approximately preserves the clustering structure. Since solving the $k$-means problem is NP-hard in general \cite{dasgupta2008hardness}, we run an approximated $k$-means algorithm \cite{song2010fast} on the projected data matrix, whose running time is polynomial in $n,k$ and $d$. There are many versions of spectral clustering algorithms and here is the pseudo code of the one used in this paper.  
\begin{itemize}
\item[1.] Compute the SVD of the data matrix $Y=[y_1,\cdots, y_n] = UDV'$. Let $U_k$ be the first $k $ columns of $U$.
\item[2.] Project $y_1,\cdots, y_n$ onto $U_k$, i.e. let $\hat{y_i} = U_kU_k' y_i$ for $i \in [n]$.
\item[3.] Run a $O(1)$-approximation algorithm \cite{song2010fast} for $k$-means problem on the columns of projected matrix $\hat{Y} = [\hat{y}_1,\cdots, \hat{y}_n]$.
\end{itemize}
We refer this algorithm as spectral clustering algorithm henceforth. A lot of progress has been made in studying theoretical properties of this spectral clustering algorithm. For more details, we refer to \cite{kannan2009spectral, awasthi2012improved} and references therein. 

\section{Clustering sub-Gaussian mixtures} \label{sec:main}
In this section, we present the results of using Lloyd's algorithm to clustering sub-Gaussian mixtures. In Section \ref{sec:2main}, we give the convergence results of Lloyd's algorithm to the symmetric, two component mixture model (\ref{eq:modeltwo}). Then we extend it to the general $k$ mixture model (\ref{eq:model}) in Section \ref{sec:kmain}. A minimax lower bound is established in Section \ref{sec:lower}. 

\subsection{Two mixtures} \label{sec:2main}
A key quantity in determining the basin of attraction of Lloyd's algorithm is the following normalized signal-to-noise ratio
$$ r = \frac{\|\theta^*\|}{\sigma \sqrt{1+\eta}}, $$
where $\eta=9d/n$. Here we normalize the signal-to-noise ratio by $\sqrt{1+\eta}$ because the statistical precision of estimating $\theta^*$ is at the order of $\sigma \sqrt{\eta}$. If $\|\theta^*\| \lesssim \sigma \sqrt{\eta}$, information theoretically we could not distinguish between two centers with positive probability, even when the labels are known. Let $A_s=\frac{1}{n} \sum_{i=1}^{n} \mathbb{I}\{\hat{z}_i^{(s)} \neq z_i \}$ be the mis-clustering rate at the iteration $s$ of Lloyd's algorithm, and $\hat{\theta}^{(s)}$ be the estimated centers at step $s$, $s=0,1,2,\cdots$. We have the following theorem that characterizes the behavior of $A_s$.

\begin{thm} \label{thm:2main}
Assume $n \ge C$ and $r \ge C$ for a sufficiently large constant $C$. For any given (data dependent) initializer satisfying
\begin{equation} \label{eq:2cond}
A_0 \le  \frac{1}{2}- \frac{2.56+\sqrt{\log r}}{r} - \frac{1}{\sqrt{n}} \quad \textrm{or} \quad \|\hat{\theta}^{(0)}-\theta^*\| \le \left(1 -\frac{4}{r} \right) \|\theta^*\|,
\end{equation}
with probability $1-\nu$, we have
\begin{equation} \label{eq:twoconst}
A_{s+1} \le  \left(A_s+\frac{8}{r^2}\right) A_s + \frac{2}{r^2} + \sqrt{\frac{4\log n}{n}}, \quad \textrm{for all } s \ge 0
\end{equation}
with probability greater than $1 - \nu - n^{-3}- 2\exp\left(-\frac{\|\theta^*\|^2}{3\sigma^2}\right)$,
and
\begin{equation} \label{eq:tworate}
A_{s} \le \exp \left( - \frac{\|\theta^*\|^2}{16\sigma^2} \right), \;\; \text{ for all } s \ge 3\log n
\end{equation}
with probability greater than $1 - \nu - 5n^{-1} -  8\exp\left(-\frac{\|\theta^*\|^2}{16\sigma^2}\right)$. Moreover, if $r \to \infty$ as $n \to \infty$, the error rate in (\ref{eq:tworate}) can be improved to $\exp \left( - (1+o(1))\frac{\|\theta^*\|^2}{2\sigma^2} \right)$.
\end{thm}

The proof of Theorem \ref{thm:2main} is given in Section \ref{sec:proof}. We will first show that after a few iterations, the mis-clustering ratio $A_s$ is sufficiently small (for example, smaller than 1/8) by proving $(\ref{eq:twoconst})$, then prove mis-clustering rate of Lloyd's algorithm is exponentially small with an exponent determined by the signal-to-noise ratio. Note that the mis-clustering rate $A_s$ only takes discrete values in $\{0, n^{-1}, 2n^{-1}, \cdots, 1\}$. If $\|\theta^*\| > 4\sigma \sqrt{\log n}$, (\ref{eq:tworate}) implies $A_s <1/n$, and thus Theorem \ref{thm:2main} guarantees a zero clustering error after $\lceil 3 \log n \rceil$ Lloyd's iterations with high probability.

By Theorem \ref{thm:2main}, the convergence of Lloyd's algorithm has three stages. In the first iteration, it escapes from a small neighborhood around $1/2$ and achieves an error rate $A_s \le 1/2-c$ for a small constant $c$. Then it has a linear convergence rate, which depends on the signal-to-noise ratio. Finally, similar to other two-stage estimators \cite{zhang2014spectral, gao2015achieving}, once the the mis-clustering rate $A_s$ is sufficiently small, it converges to the optimal statistical precision of the problem after one or more iterations.

As we shall see in Section \ref{sec:lower}, a necessary condition for consistently estimating the labels is $\|\theta^*\|/\sigma \to \infty$. Therefore, our condition on the initializer is just slightly stronger than the random initialization error $1/2$. Balakrishnan et al. \cite{balakrishnan2014statistical} studied the convergence of EM algorithm under the same model. They require an initializer $\hat{\theta}^{(0)} \in \mathbb{B}(\theta^*, \|\theta^*\|/4)$ under the assumption that $r$ is sufficiently large. Here we replace $1/4$ by a factor close to $1$ in Theorem \ref{thm:2main}. More specifically, the factor in the initialization condition is determined by the signal-to-noise ratio. The stronger the signal-to-noise ratio is, the weaker initialization condition we need.  

\subsection{$k$ mixtures} \label{sec:kmain} 
Now we consider general number of mixtures. To better present our results, we first introduce some notation. For all $h \in [k]$, let $T_h^*$ be the true cluster $h$ and $T_h^{(s)}$ be the estimated cluster $h$ at iteration $s$. Define $n_h^* = |T_h^*|$, $n_h^{(s)}=|T_h^{(s)}|$ and $n_{hg}^{(s)}=|T_h^* \cap T_g^{(s)}|$. The mis-clustering rate at iteration $s$ can be written as 
\[ A_s = \frac{1}{n} \sum_{i=1}^{n} \mathbb{I}\{\hat{z}_i^{(s)} \neq z_i\} = \frac{1}{n} \sum_{g \neq h \in [k]^2} n_{gh}^{(s)}.  \]
We define a cluster-wise mis-clustering rate at iteration $s$ as
\[ G_{s} = \max_{h \in [k]}\left\{\frac{\sum_{g \neq h \in [k]} n_{gh}^{(s)}}{n_h^{(s)}}, \frac{\sum_{g \neq h \in [k]} n_{hg}^{(s)}}{n_h^*}   \right\}. \]
The first term in the maximum operator of definition of $G_s$ can be understood as the false positive rate of cluster $h$ and the second term is the true negative rate of cluster $h$. It is easy to see the relationship that $A_s \le G_s$.

Let $\Delta = \min_{g \neq h \in [k]} \|\theta_g- \theta_h\|$ be the signal strength. For $h \in [k]$, let $\hat{\theta}^{(s)}_h$ be the estimated center of cluster $h$ at iteration $s$. Define our error rate of estimating centers at iteration $s$ as
\[ \Lambda_s = \max_{h \in [k]} \frac{1}{\Delta} \|\hat{\theta}^{(s)}_h - \theta_h\|. \]
Besides signal-to-noise ratio, there are other two factors that determine the convergence of Lloyd's algorithm, the maximum signal strength and the minimum cluster size. As argued by \cite{jin2016local}, when one cluster is very far away from other clusters, local search may fail to find the global optimum, which indicates initialization should depend on the maximum signal strength.  Since the cluster sizes influence the accuracy to estimate centers, we a lower bound on the size of the smallest clusters. Define $\lambda = \max_{g \neq h \in [k]} \|\theta_g- \theta_h\|/\Delta$ and $\alpha=\min_{h \in k} n_h^*/n$. Similar to the two-cluster case, we define a normalized signal-to-noise ratio
\begin{equation}
r_k = \frac{\Delta}{\sigma} \sqrt{\frac{\alpha}{1+kd/n}}.
\end{equation}
Now we are ready to present our results for the $k$ mixtures. 

\begin{thm} \label{thm:kmeans}
Assume $n\alpha^2  \ge C k \log n  $ and $ r_k \ge C \sqrt{k} $ for a sufficiently large constant $C$. Given any (data dependent) initializer satisfying
\begin{equation} \label{eq:kcond}
G_0  < \left( \frac{1}{2} - \frac{6}{\sqrt{r_k}} \right)\frac{1}{\lambda} \quad \textrm{or} \quad \Lambda_0 \le \frac{1}{2} - \frac{4}{\sqrt{r_k}},
\end{equation}
with probability $1-\nu$, we have
\begin{equation} \label{eq:kconst}
G_{s+1} \le  \frac{C_1}{r_k^2} G_s + \frac{C_1}{r_k^2} + \sqrt{\frac{5k \log n}{\alpha^2 n}}  \quad \textrm{for all } s \ge 1
\end{equation}
for some constant $C_1 \le C$ with probability greater than $1 -\nu -n^{-3}$,  and
\begin{equation} \label{eq:efficiency}
A_{s} \le \exp \left( - \frac{\Delta^2}{16\sigma^2} \right) \;\; \text{ for all } s \ge 4\log n 
\end{equation}
with probability greater than $1- \nu - 4/n-2\exp \left( - \Delta/\sigma \right)$. Moreover, if $\sqrt{k}=o(r_k)$ and $k \log n =o(n\alpha^2)$ as $n \to \infty$, the exponent in (\ref{eq:efficiency}) can be improved to $\exp \left( - (1+o(1)) \frac{\Delta^2}{8\sigma^2} \right)$. 
\end{thm}

In Theorem \ref{thm:kmeans}, we establish similar results as in the two-mixture case. The initialization condition (\ref{eq:kcond}) here is slightly stronger due to the asymmetry. Similarly, after initialization that satisfies (\ref{eq:kcond}), the convergence has three stages. It first escapes from a small neighborhood of $1/2$, then converges linearly, and finally it achieves an exponentially small mis-clustering rate after $\lceil 4\log n \rceil$ iterations.

From Theorem \ref{thm:kmeans}, Lloyd's algorithm does not require to know or to estimate the covariance structure of Gaussian mixtures. Likelihood-based algorithms, such as EM algorithm and methods of moments, need to estimate the covariance structure. Lloyd's algorithm works for any general sub-Gaussian mixtures satisfying (\ref{eq:subgaussian}), and thus is robust to different models. This demonstrates a key advantage of using Lloyd's algorithm than other likelihood-based algorithms. 

Theoretically, the initialization condition (\ref{eq:kcond}) in Theorem \ref{thm:kmeans} is sharp in a sense that we give a counterexample in Section \ref{sec:counter} showing that Lloyd's algorithm may not converge when $G_0=1/(2\lambda)$ or $\Lambda=1/2$. 

Now let us give a sufficient condition for the spectral clustering initializer to fall into the basin of attraction (\ref{eq:kcond}). By Claim 1 in Section 3.2 of \cite{kannan2009spectral} and Lemma \ref{lm:tech4}, for each center $\theta_h$, there is a center $\mu_h$ returned by spectral clustering algorithm such that $\|\mu_h-\theta_h\| \lesssim (\sqrt{k}/r_k)\Delta$ with probability greater than $1-\exp(-n/4)$, which implies the following corollary. 

\begin{corollary} \label{cor:main}
Assume $\sqrt{k} = o(r_k)$  and $ k \log n =o(n\alpha^2)$ as $n \to \infty$. Let $\hat{z}$ be the output of Lloyd's algorithm initialized by spectral clustering algorithm after $\lceil 4\log n \rceil$ iterations. Then
\begin{equation}
\ell(\hat{z}, z) \le \exp \left( - (1+o(1)) \frac{\Delta^2}{8\sigma^2} \right)
\end{equation}
with probability greater than $1-5n^{-1}-2\exp \left( - \Delta/\sigma \right)$.
\end{corollary}

This corollary gives a sufficient separation (signal-to-noise ratio) condition for clustering sub-Gaussian mixtures. When $n \ge kd$ and all the cluster sizes are of the same order, our separation condition simplifies to $\Delta \gtrsim k \sigma $. When there are finite number of clusters ($k$ is finite), our separation condition further simplifies to $\Delta \gtrsim \sigma$. To our knowledge, this is the first result to give an explicit exponentially small error rate for estimating the labels. Previous studies mostly focus on exact recovery of the cluster labels and have no explicit statistical convergence rate. Furthermore, our results hold for all the range of $d$, while previous results all require $n \gg kd$. 

Previously, the best known separation condition on the Lloyd-type algorithm is $r_k \gtrsim \sqrt{k} \textrm{ poly}\log n$, under the assumption that $n \gg kd/\alpha$  \cite{awasthi2012improved}. Corollary \ref{cor:main} improves it by a $\textrm{poly} \log n$ factor. For Gaussian mixtures, among all algorithms including Lloyd's algorithm, the best known separation condition is $\Delta \gtrsim \sigma(\alpha^{-1/2}+\sqrt{k^2+k \log(nk)})$ under the assumption that $\alpha n \gtrsim k(d+\log k)$ \cite{achlioptas2005spectral}. Our condition is weaker when the cluster sizes are of the same order. When $k$ is finite, our condition is weaker by a $\sqrt{\log n}$ factor.  As we shall see in Section \ref{sec:data}, $k$ is usually small in practice and can often be regarded as a constant. 



\subsection{Minimax lower bound} \label{sec:lower}
To show that the mis-clustering rate rate in Corollary \ref{cor:main} cannot be improved, we present a rate matching lower bound in this section. Define a parameter space as follows,
\begin{equation}
    \begin{split}
 \Theta =    \left\{ (\theta, z), ~ \theta=[\theta_1,\cdots,\theta_k] \in \mathbf{R}^{d \times k}, ~ \Delta \le \min_{g \neq h} \|\theta_g - \theta_h\|, \right. \\
    \left. z: [n] \to [k], ~|\{i \in [n], z_i = u\}| \ge \alpha n, \forall u \in [k] \vphantom{\int_1^2} \right\}.
    \end{split}
\end{equation}
We have the following minimax lower bound. 
\begin{thm} (Lower Bound) \label{thm:lower_bound} 
For model (\ref{eq:model}), assume independent gaussian noise $w_{ij} \stackrel{i.i.d}{\sim}\mathcal{N}(0,\sigma^2)$, then when $\frac{\Delta}{\sigma \log (k/\alpha)} \to \infty$, 
\[ \inf_{\hat{z}} \sup_{(z,\theta) \in \Theta} \E \ell(\hat{z}, z)  \ge \exp \left( - (1+o(1))\frac{\Delta^2}{8\sigma^2} \right). \]
If $\frac{\Delta}{\sigma} + \log (k/\alpha) = O(1)$, $\inf_{\hat{z}} \sup_{(z,\theta) \in \Theta} \E \ell(\hat{z}, z)\ge c$ for some constant $c>0$.
\end{thm}

This lower bound result shows that if the signal-to-noise ratio $\Delta/\sigma$ is at the constant order, the worst case mis-clustering rate is lower bounded by a constant. In other words, a necessary condition for consistently estimating the labels is $\Delta/\sigma \to \infty$. To achieve strong consistency, we need $\Delta/\sigma \gtrsim \sqrt{\log n}$. Theorem \ref{thm:lower_bound} indicates the necessity of separation conditions in estimating the labels. Previous results that use methods of moments to estimate the cluster centers \cite{hsu2013learning, anandkumar2012method}, on the contrary, do not require separation conditions. This reveals the difference between estimating the cluster centers and the labels. 

Together with Corollary \ref{cor:main}, Theorem \ref{thm:lower_bound} gives us the minimax rate of estimating the underlying labels of Gaussian mixtures. 

\section{Applications} \label{sec:app}
In this section, we generalize the results in Section \ref{sec:main} to community detection and crowdsourcing by considering two variants of Lloyd's algorithm. 
 
\subsection{Community Detection} \label{sec:SBM}
We have a network of $k$ communities. There are $n$ nodes in the network and $\{z_i \in [k], i \in [n]\}$ is the community assignment. We observe a symmetric adjacency matrix $A \in \{0,1\}^{n \times n}$ with zero diagonals, which is generated by the following stochastic block model \cite{holland1983stochastic} (SBM thereafter),
$$ A_{ij} \sim \begin{cases}
    \text{Ber}(\frac{a}{n}),& \text{if } z_i=z_j\\
    \text{Ber}(\frac{b}{n}), & \text{otherwise,}
\end{cases} $$
for $1 \le i < j \le n$, where $0 < b < a <n$ and $\{z_i\}$ are unknown parameters. 
The SBM is a special case of the Planted Partition Model \cite{mcsherry2001spectral} in the theoretical computer science literature. As in \cite{gao2015achieving}, we use a single parameter $\beta$ to control the community sizes, namely, there is a $\beta>0$ such that $\frac{n}{\beta k} \le | \{i \in [n], z_i = g\} |$ for all $g \in [k]$.  The goal of community detection is to estimate community structure $z$ using the adjacency matrix $A$. Note that
$\E A_{i \cdot} $ only takes $k$ different values, and the SBM can be viewed as a special case of model  (\ref{eq:model}) with $d=n$ and $\sigma^2=\frac{a}{n}$, ignoring the fact that $A$ is symmetric with zero diagonals.


When $a$ and $b$ are of an order of $n$, the SBM satisfies the sub-Gaussian noise condition (\ref{eq:subgaussian}) and we have
$\Delta^2 \ge \frac{(a-b)^2}{n^2} \times \frac{2n}{\beta k} = \frac{2  (a-b)^2}{\beta k n}$
and $\sigma^2 \le a/n$.  
Initialized by the spectral clustering algorithm, let $\hat{z}$ be the labels returned by running Lloyd's algorithm on the rows of adjacency matrix $A$ with $\lceil 4\log n \rceil $ iterations. By considering the fact that $A$ is symmetric with zero diagonals, we can slightly modify the proof of Corollary \ref{cor:main} to obtain the following result. 
\begin{corollary} \label{thm:SBM}
Assume $\frac{(a-b)^2}{\beta^3 a k^4} \to \infty$ and $\frac{n}{\beta^2 k^3 \log n} \to \infty$ as $n \to \infty$, then
\begin{equation} \label{eq:sbm.rate}
\ell(z, \hat{z}) \le \exp \left( -(1+o(1)) \frac{(a-b)^2}{4a \beta k} \right) 
\end{equation}
with probability tending to $1$.
\end{corollary}

We omit the proof of Corollary \ref{thm:SBM}. A more interesting case of the SBM is the sparse case when $a=o(n)$ and $b=o(n)$, in which the sub-Gaussian assumption (\ref{eq:subgaussian}) does not hold. By utilizing the community structure, we consider the following variants of Lloyd's algorithm. 

\begin{algorithm}
\caption{CommuLloyd} \vspace{0.05in}
\textbf{Input}: Adjacency matrix $A$. Number of communities $k$.\vspace{0.02in}. Trimming threshold $\tau$.  \\
\textbf{Output}: Estimated labels $\hat{z}_1, \cdots, \hat{z}_n$. \vspace{0.02in} \\
1. Trim the adjacency matrix: 
\begin{itemize}
\item[1a.] calculate the degree of each node $d_i = \sum_{j=1}^{n} A_{ij}, \forall i \in [n]$. \vspace{-0.01in}
\item[1b.] trim the adjacency matrix $A_{ij}^\tau = A_{ij} \mathbb{I}\{d_i \le \tau \}, \forall i \in [n], j \in [n]$.
\end{itemize}
2.  Run spectral clustering algorithm on the trimmed matrix $A_{\tau}$. \\
3. Run following iterations until converge,
\begin{eqnarray} 
\label{eq:SBMcenter}  \hat{B}_{i h} &=& \frac{\sum_{j=1}^{n} A_{ij} \mathbb{I}\{\hat{z}_j=h\}}{\sum_{j=1}^{n} \mathbb{I}\{\hat{z}_j=h\}}, \quad \forall~h \in [k], i \in [n]. \\
\label{eq:SBMlabel}  \hat{z}_i &=&  \argmax_{h \in [k]} \hat{B}_{i h}, \quad \forall~ i \in [n].
\end{eqnarray}
\label{alg:Commulloyd}
\end{algorithm}

As pointed out by \cite{chin2015stochastic}, under the sparse setting of $a \lesssim \log n$, the trimmed adjacency matrix $A^{\tau}$ is a better estimator of $\E A$ than $A$ under the spectral norm. Therefore, we run spectral clustering on the $A^{\tau}$ to get an initial estimator of $z$. Then we update the labels node by node. For the $i$-th node in the network, we estimate its connecting probability to community $j$ (defined as $B_{z_i j}$) based on our previous estimated labels. Then we assign its label to be the $j$ that maximizes $\hat{B}_{z_i j}$, which is expected to be close to $a/n$ or $b/n$, depending on whether $z_i=j$. 
The following theorem gives a theoretical guarantee of the CommuLloyd algorithm. 

\begin{thm} \label{thm:SBMyu}
Assume $n \ge 6k^2 \log n$, $(a-b)^2 \ge C_0 a \beta^2 k^2 \log (\beta k)$  for a sufficiently large constant $C_0$. Let $A_s$ be the mis-clustering rate at iteration $s$ of the CommuLloyd algorithm. Then for any initializer satisfies $G_0 \le 1/4$ with probability $1-\nu$, we have
\begin{equation} \label{eq:SBMthm} 
A_{s+1} \le \exp \left( - \frac{(a-b)^2}{2C\beta a k} \right) + \frac{4}{5} A_s, \quad \forall~ 1 \le s \le 3\log n
\end{equation}
with probability greater than $1-n^{-1}-\nu$.
\end{thm}

The proof of Theorem \ref{thm:SBMyu} is given in Section \ref{sec:SBMproof} of the appendix. We show the linear convergence of CommuLloyd algorithm given the first stage group-wise mis-clustering rate $G_0 \le 1/4$. In fact, this initialization assumption can be relaxed to $G_0 \le 1/2 - \sqrt{a \log(\beta k)} \beta k/ (a-b)$. To better present our results, we simplify it to $1/4$. By Lemma 7 in \cite{gao2015achieving} and Theorem 3.1 in \cite{awasthi2012improved}, we can guarantee a group-wise initialization error of $1/4$ when $(a-b)^2 \gtrsim a\beta^2 k^3$. Therefore, we obtain an exponentially small error rate under the signal-to-noise ratio condition
\[ (a-b)^2 \gtrsim a \beta^2 k^3 \log \beta. \]

Theorem \ref{thm:SBMyu} provides theoretical justifications of the phenomenon observed in the numerical experiments of \cite{gao2015achieving}. More iterations achieve better mis-clustering rate. The Lloyd iterations enjoy similar theoretical performance as the likelihood based algorithm proposed in \cite{gao2015achieving}. While they require a global initialization error to be $o(1/(k \log k))$,  we require a cluster-wise initialization error to be smaller than $1/4$. Moreover, the CommuLloyd algorithm is computationally much more efficient than Algorithm 1 in \cite{gao2015achieving}, which requires obtaining $n$ different initializers and hence running SVD on the network $n$ times. Theoretically, we relax the assumption in \cite{gao2015achieving} that $a \asymp b$ and $\beta$ is a constant, and we improve the best known signal-to-noise ratio condition\cite{gao2015achieving} by a $\log k$ factor.  When $\beta$ is a constant and $(a-b)^2/(ak^3) \to \infty$ as $n \to \infty$ , we are able to match the minimax rate obtained in \cite{zhang2015minimax}.

\subsection{Crowdsourcing}
Crowdsourcing is an efficient and inexpensive way to collect a large amount of labels for supervised machine learning problems. We refer to \cite{zhang2014spectral, gao2016exact} and references therein for more details. The following Dawid-Skene model is the most popular model considered in the crowdsourcing literature. 

Suppose there are $m$ workers to label $n$ items. Each item belongs to one of the $k$ categories. Denote by $z= (z_1, z_2, \cdots, z_n) \in [k]^{n}$ the true labels of $n$ items and by $X_{ij}$ the label of worker $i$ given to item $j$. Our goal is to estimate the true labels $z$ using $\{X_{ij}\}$. Dawid-Skene model assumes that the workers are independent and that given $z_j = g$, the labels given by worker $i$ are independently drawn from a multinomial distribution with parameter $\pi_{ig*}= \left(\pi_{ig1}, \cdots, \pi_{igk} \right)$, i.e.,
\[ \mathbb{P} \left\{ X_{ij} = h \vert z_j = g \right\} = \pi_{igh}\]
for all $g, h \in [k], i \in [m], j \in [n]$. 
Dawid-Skene model can be seen as a special case of the mixture model (\ref{eq:model}) via the following transformation. For $j \in [n]$, let 
$$y_j = \Big( \mathbb{I}{\{X_{1j}=1\}}, \cdots,  \mathbb{I}{\{X_{1j}=k\}}, \cdots,  \mathbb{I}{\{X_{mj}=1\}}, \cdots,  \mathbb{I}{\{X_{mj}=k\}}  \Big)'.$$
Then given $z_j=g$, we have $\E [y_j \vert z_j=g] = \theta_g$ with $\theta_g =\left( \pi_{1g1}, \pi_{1g2}, \cdots, \pi_{mgk}  \right)$.
By defining $w_{ijh}=\mathbb{1}\{X_{ij}=h\}-\pi_{i z_j h}$, our observations $y_j$ can be decomposed as the following signal-plus-noise model (\ref{eq:model}) with $$w_j = [ w_{1j1}, w_{1j2}, \cdots, w_{mjk}]'.$$ 

Therefore, we consider the following variants of the Lloyd's algorithm in Algorithm \ref{alg:crowdlloyd}. We iteratively estimate workers' accuracy and items' labels. Iterations (\ref{eq:CSworker}) and (\ref{eq:CSlabel}) are actually Lloyd's iteration (\ref{eq:centerupdate}) and (\ref{eq:labelupdate}) under the transformed Dawid-Skene model. Iterations (\ref{eq:CSworker}) and (\ref{eq:CSlabel}) are also closely related to the EM update for Dawid-Skene model \cite{zhang2014spectral}. (\ref{eq:CSworker}) is the same as the M-step. In (\ref{eq:CSlabel}), we first use least squares to approximate the log-likelihood function of multinomial distributions, then we do a hard labeling step instead of soft labeling in the E-step of EM algorithm. 
\begin{algorithm}
\caption{CrowdLloyd} \vspace{0.05in}
\textbf{Input}: $\{X_{ij}, i \in [m], j \in [n]\}$. Number of possible labels $k$. \vspace{0.02in}  \\
\textbf{Output}: Estimated labels $\hat{z}_1, \cdots, \hat{z}_n$. \vspace{0.02in} \\
\text{1. Initialization via Majority Voting}:
\begin{equation} \label{eq:mvoting} 
\hat{z}_j = \argmax_{g \in [k]} \sum_{i=1}^{m} \mathbb{I}\{X_{ij}=g\} \quad \forall~ j \in [n].
\end{equation}
2. Run following iterations until converge.
\begin{eqnarray} 
\label{eq:CSworker} 
\hat{\pi}_{igh} &=& \frac{\sum_{j=1}^{n} \mathbb{I}\{X_{ij}=h, \hat{z}_j=g\}}{\sum_{h=1}^{k} \sum_{j=1}^{n} \mathbb{I}\{X_{ij}= h, \hat{z}_j=g\}} \quad \forall~ i \in [m], g,h \in [k]^2. \\
 \label{eq:CSlabel} 
\hat{z}_j &=&  \argmin_{h \in [k]} \left[ \sum_{i=1}^{m} \sum_{g=1}^{k} \left( \mathbb{I}\{X_{ij}=g\} - \hat{\pi}_{igh} \right)^2 \right] \quad \forall~ j \in [n].
\end{eqnarray}
\label{alg:crowdlloyd}
\end{algorithm}

As suggested by Corollary \ref{cor:main}, we could use spectral clustering to initialize the labels. But the majority voting initializer, as described in the first step of Algorithm \ref{alg:crowdlloyd}, is a more natural and commonly used  initializer of crowdsourcing. By regarding all the workers have the same accuracy, it simply estimates the labels by aggregating the results of every worker with equal weights. The majority voting initializer is computationally more efficient than spectral methods. We have the following result of the majority voting estimator. 

\begin{thm} \label{thm:CSmvoting}
Under the Dawid-Skene model, the majority voting initializer satisfies
\begin{equation}
\frac{1}{n} \sum_{j=1}^{n} \mathbb{I}\left\{ \hat{z}_j \neq z_j \right\} \le \exp \left( - \frac{mV(\pi)}{2} \right)
\end{equation}
with probability greater than $1-\exp\left( - m V(\pi)/2 + \log k \right)$, where 
\[  V(\pi) = \min_{g \neq h} \frac{1}{m} \left(\sqrt{\sum_{i=1}^{m}\pi_{igg}} - \sqrt{\sum_{i=1}^{m}\pi_{igh}}\right)^2. \]
\end{thm}
The proof of Theorem \ref{thm:CSmvoting} is given in Section \ref{sec:CSproof}. To our knowledge, this theorem is the first theoretical characterization of the Majority Voting estimator under the general Dawid-Skene model. $mV(\pi)$ measures the collective accuracy of $m$ workers on estimating one of the labels. If there are two labels $g$ and $h$ that are confusing for most of the workers, the value of $V(\pi)$ is small and majority voting estimator may not have good performance.  Under the special case of one-coin model, where $\pi_{igg} = p$ and $\pi_{igh} = \frac{1-p}{k-1}$ for $g \neq h$, we have $$mV(\pi) = m(\sqrt{p} - \sqrt{(1-p)/(k-1)})^2 \gtrsim \log m$$ as long as $p \ge \frac{1}{k}+\sqrt{\frac{\log m}{m}}$. That is, for Majority Voting estimator to be consistent, the accuracy of workers giving true labels only needs to be $\sqrt{\log m/m}$ better than random guess.

With Theorem \ref{thm:CSmvoting} and Theorem \ref{thm:kmeans}, we are able to to obtain an error rate upper bound of Algorithm \ref{alg:crowdlloyd}. It can be proved that the sub-Gaussian parameter of $w_i$ is 2 (see the proof of Corollary \ref{thm:crowdsourcing}
for more details). Define $\Delta^2 = \min_{g \neq h} \sum_{i=1}^{m} \sum_{u=1}^{k} (\pi_{igu}-\pi_{ihu})^2.$ Then we have the following corollary of Theorem \ref{thm:kmeans}.
\begin{corollary} \label{thm:crowdsourcing}
Assume $mV(\pi) \ge C \log k$, $\alpha \Delta^2 \ge  C k $ and $n\alpha \ge k \Delta^2 \log n$ for a sufficiently large constant $C$. Let $\hat{z}=(\hat{z}_1, \cdots, \hat{z}_n)$ be the estimated labels returned by running Algorithm \ref{alg:crowdlloyd}. We have
\begin{equation} \label{eq:crowd}
\frac{1}{n} \sum_{i=1}^{n} \mathbb{I} \left\{ \hat{z}_i \neq z_i  \right\} \le \exp \left( - \frac{\Delta^2}{64} \right)
\end{equation}
with probability greater than $1- \exp(-mV(\pi)/4) - 4/n-2\exp \left( - \Delta/2 \right)$.
\end{corollary}

The proof of Theorem \ref{thm:crowdsourcing} is given in Section \ref{sec:CSproof}. We achieve an exponentially small error rate for estimating the labels under Dawid-Skene model. By Theorem 4.2 in \cite{gao2016exact}, we can achieve the minimax optimal rate by an additional likelihood based Neyman-Pearson testing step followed by Algorithm \ref{alg:crowdlloyd}. Previous provable results \cite{zhang2014spectral, gao2016exact} assume $\pi_{igh}\ge \rho > 0$ for all $(i,g,h) \in [m] \times [k]^2$, which is quite restrictive since a very good worker could have $\pi_{igh}=0$ for some $h \neq g$. By considering least square update (\ref{eq:CSlabel}), we successfully get rid of this assumption. In contrast to algorithms that using spectral methods as initializer \cite{zhang2014spectral}, Theorem \ref{thm:crowdsourcing} does not need an eigenvalue condition on the confusion matrix. 

\section{Numerical Analysis} \label{sec:data}
\subsection{Gaussian Mixture Model}
\subsubsection{Simulation}
In this section, we provide simulation results that are designed to test our theoretical findings in Theorem \ref{thm:kmeans}. As predicted by  Theorem \ref{thm:kmeans}, there are three stages of the convergence of Lloyd's algorithm. Given an initial estimator of the labels with error smaller than $1/(2\lambda)$, Lloyd's algorithm escapes from a small neighborhood of $1/(2\lambda)$ in the first iteration. Then it has a geometrically decay mis-clustering rate during the second stage. Finally, it converges to the statistical precision of the problem and achieves an exponentially small error. In all the results reported below, we plot the logarithm of the mis-clustering rate versus the iteration count. Each curve plotted is an average of $10$ independent trails. 

\begin{figure}[h]
\centering
\begin{subfigure}{.5\textwidth}
\centering
\includegraphics[width=1.0\linewidth]{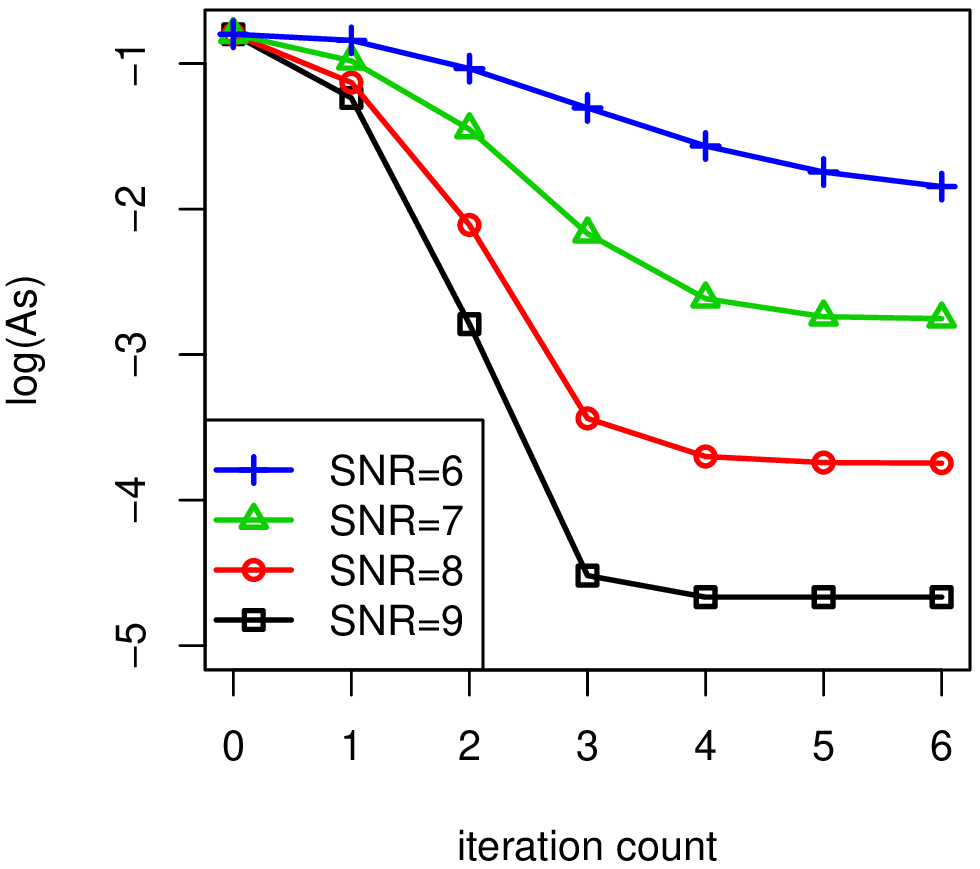}
\caption{}
\label{fig:gmmsnr}
\end{subfigure}%
\begin{subfigure}{.5\textwidth}
\centering
\includegraphics[width=1.0\linewidth]{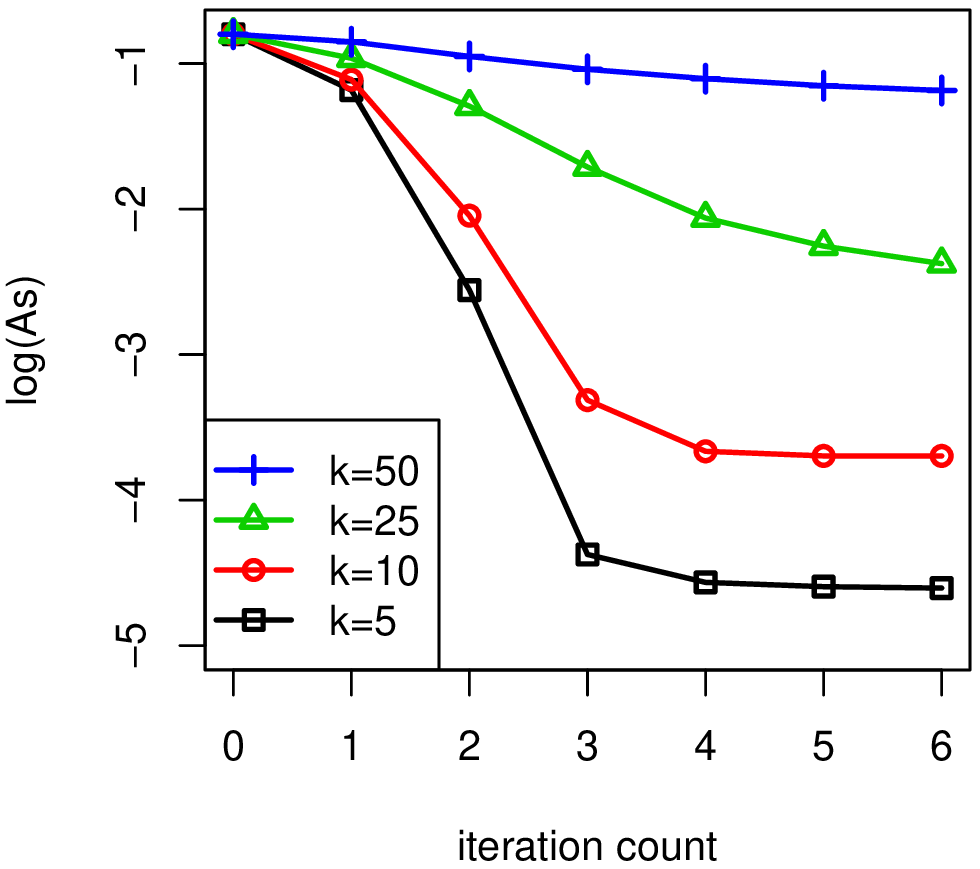}
\caption{}
\label{fig:gmmk}
\end{subfigure}%
\caption{Plots of the log mis-clustering rate $A_s$ versus iteration count.}
\label{figure:gmm}
\end{figure}

In the first simulation, we independently generate $n=1000$ samples from a mixture of $k=10$ spherical Gaussians. Each cluster has $100$ data points. The centers of those $10$ gaussians are orthogonal unit vectors in $\mathbb{R}^d$ with $d=100$. The four curves in Figure \ref{fig:gmmsnr} corresponds to four choices of the standard deviation $\sigma = 2/SNR$ with  $SNR \in \{6, 7, 8, 9\}$.  We use the same initializer with a group-wise mis-clustering rate of $0.454$ for all the experiments. As shown by plots in Figure \ref{fig:gmmsnr}, there is a linear convergence of the log mis-clustering rate at the first 2 iterations and then the convergence slows down. After four iterations, the log mis-clustering rate plateaus roughly at the level of $-\frac{1}{16} SNR^2$. Moreover, we observe in Figure \ref{fig:gmmsnr} that the convergence rate increases as $SNR$ increases, as expected from Theorem \ref{thm:kmeans}.

The second experiment investigates the role of $k$ in the convergence of Lloyd's algorithm. We first randomly generate $k \in \{5,10,25,50\}$ centers that are orthogonal unit vectors in $\mathbb{R}^{100}$. Then we add independent gaussian noise $\mathcal{N}(0, 0.25^2)$ to each coordinate of those centers to get $1000/k$ samples for each cluster. In total, we have $n=1000$ independent samples from a mixture of spherical Gaussian distributions. We set the same initializer as in the first simulation to make a group-wise mis-clustering rate of $0.37$. As observed in Figure \ref{fig:gmmk}, the larger the $k$, the slower Lloyd's algorithm converges. This agrees with the phenomena predicted by Theorem \ref{thm:kmeans}. It is interesting to observe that when $k$ is large, Lloyd's algorithm does not converge to the optimal mis-clustering rate in Figure \ref{fig:gmmk}. When the value of $k$ is doubled, the convergence rate of $\log A_s$ slows down roughly by half, which matches the $k$-term in (\ref{eq:kconst}).

\subsubsection{Gene Microarray Data}
To demonstrate the effectiveness of Lloyd's algorithm on high dimensional clustering problem, we use six gene microarray data sets that were preprocessed and analyzed in \cite{dettling2004bagboosting,gordon2002translation}. In each data set, we are given measurements of the expression levels of $d$ genes for $n$ people.  There are several underlying classes (normal or diseased) of those people. Our goal is to cluster them based on their gene expression levels, with the number of groups $k$ given. One common feature of those six data sets is that the dimensionality $d$ is much larger than the number of samples $n$. As shown in Table \ref{table:gene}, while $n$ is usually smaller than two hundreds, the number of genes can be as large as ten thousands. 

We compare three different methods on these six data sets. RandLloyd is Lloyd's algorithms initialized by randomly drawing $k$ rows of the original data as $k$ centers. SpecLloyd runs RandLloyd on the first $k$ eigenvectors of the data, followed by Lloyd's iterations on the original data. IF-HCT-PCA, proposed in \cite{jin2014important}, does a variable selection step first and then applies spectral clustering on selected variables. 
All three methods need random initializations at the $k$-means step and thus have algorithmic randomness. We randomly initialize centers for 30 times and use the initializer that give us the minimum $k$-means objective value. 

\begin{table} 
\begin{center}
\begin{tabular}{ | c | c c c| c  c c |}
\hline
Date set & k & n & d &IF-HCT-PCA  & SpecLloyd & RandLloyd  \\ \hline
Brain & 5 & 42 & 5597 & .262   & \textbf{.119} & .262(.04) \\ 
Leukemia & 2 & 72 & 3571  & .069   & \textbf{.028} & .\textbf{028} \\ 
Lung(1) & 2 & 181 & 12533 & \textbf{.033}   & .083 & .098 \\ 
Lymphoma & 3 & 62 &4026  & .069   & .177 & \textbf{.016} \\ 
Prostate& 2 & 102 & 6033 & \textbf{.382}   & .422 & .422 \\ 
SRBCT & 4 & 63 & 2308 & \textbf{.444}   & .492 & .492 \\ \hline
\end{tabular}
\end{center}
\caption{Clustering error rate of three clustering methods on six gene microarray datasets.}  \label{table:gene}
\end{table}

In Table \ref{table:gene}, we report the average mis-clustering rate rates over 30 independent experiments. We also report the standard deviation if it is larger than .01. As we can see, RandLloyd is comparable to other two methods on these gene microarray data. It achieves the best clustering error rate when the clustering task is relatively easy (signal-to-noise ratio is strong), for example, Leukemia and Lymphoma data sets. Note that both IF-HCT-PCA and SpecLloyd have a dimension reduction step. This agrees with the phenomena suggested by Theorem \ref{thm:kmeans} that Lloyd's algorithm also works on the high dimensional setting when the signal-to-noise ratio is strong enough.

\subsection{Community Detection}
In this section, we report the results of empirical studies to compare the CommuLloyd algorithm proposed in Section \ref{sec:SBM} with other methods. In particular, we compare to spectral clustering (Spec) and spectral clustering with one step refinement studied in \cite{gao2015achieving}. Since the algorithm in \cite{gao2015achieving} is motivated by the maximum likelihood estimator, we refer it to as Spec+MLE.
\subsubsection{Simulation}
Using the same settings as in \cite{gao2015achieving}, we generate data from SBM under three different scenarios: (1) dense network with equal size communities; (2) sparse network with equal size communities; and (3) dense network with unequal size communities. The detailed set-up of three simulations is described in Table \ref{table:SBMsimu}. We conduct each experiment with 10 independent repetitions and report the average results of each algorithm in Figure \ref{figure:SBM}.

For all the experiments, we use spectral clustering methods in \cite{lei2013consistency} as our initializer and plot the logarithm mis-clustering rate versus the Lloyd iteration or MLE iteration counts. Iteration 0 represents the spectral clustering initializer. As shown by Figure \ref{figure:SBM}, CommuLloyd significantly outperforms spectral clustering and spectral clustering with one step MLE refinement under all three settings. Surprisingly, it is even better than multiple steps MLE refinements. In panel (a), MLE refinements converge to a sub-optimal mis-clustering rate, and in panel (b), it does not improve the spectral clustering initializer. This phenomenon may be due to the stronger initialization condition for MLE refinements in \cite{gao2015achieving}, while in Theorem \ref{thm:SBMyu}, CommuLloyd needs only a constant-error initializer to converge to the optimal mis-clustering rate rate. The convergence rate of CommuLloyd for the these two cases, however, is faster than the linear convergence predicted by Theorem \ref{thm:SBMyu}. As for the unbalance case in panel (c), CommuLloyd and MLE has similar performance and they both exhibit a linear convergence rate. 

\begin{table} 
\begin{center}
\begin{tabular}{ |c | c c c c c |}
\hline
Setting  & n & k & Within & Between & Community Size \\ \hline
Balanced Case & 2000 & 10  & 0.20 & 0.11 & equal \\ \hline
Sparse Case & 2000 & 4 & 0.019 & 0.005 & equal \\ \hline
Unbalanced Case & 1000 & 4 & 0.35 & 0.22 & 100,200,300,400 \\ \hline
\end{tabular}
\end{center}
\caption{The settings of three simulations on community detection. Within is the within-community connecting probability $a/n$. Between is the between-community connecting probability $b/n$. Equal means the communities are of equal size. }
\label{table:SBMsimu}
\end{table}

\begin{figure}[h]
\centering
\begin{subfigure}{.33\textwidth}
\centering
\includegraphics[width=1.0\linewidth]{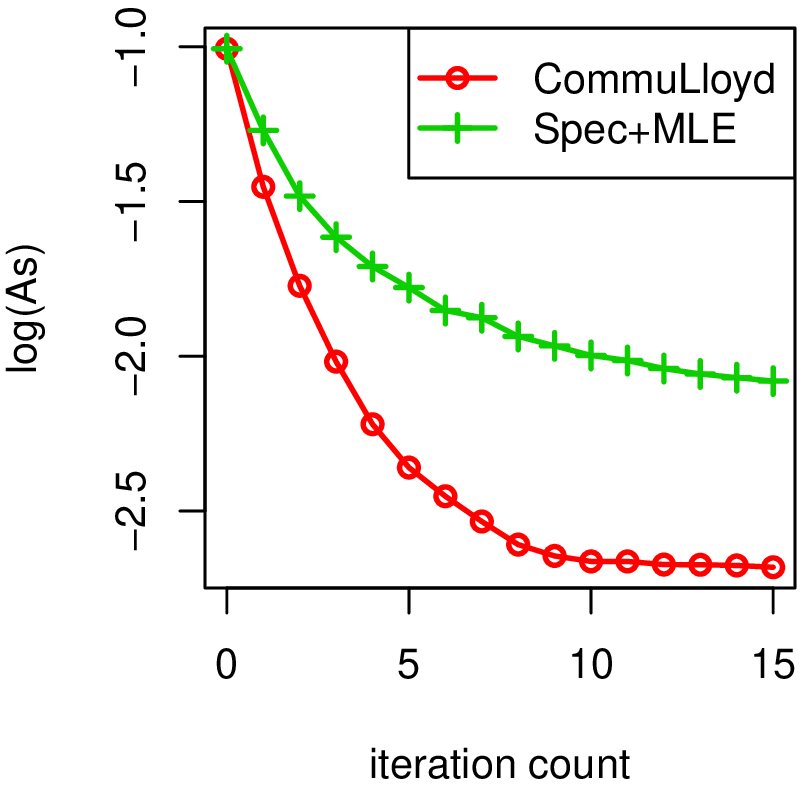}
\caption{}
\label{fig:SMBba}
\end{subfigure}%
\begin{subfigure}{.33\textwidth}
\centering
\includegraphics[width=1.0\linewidth]{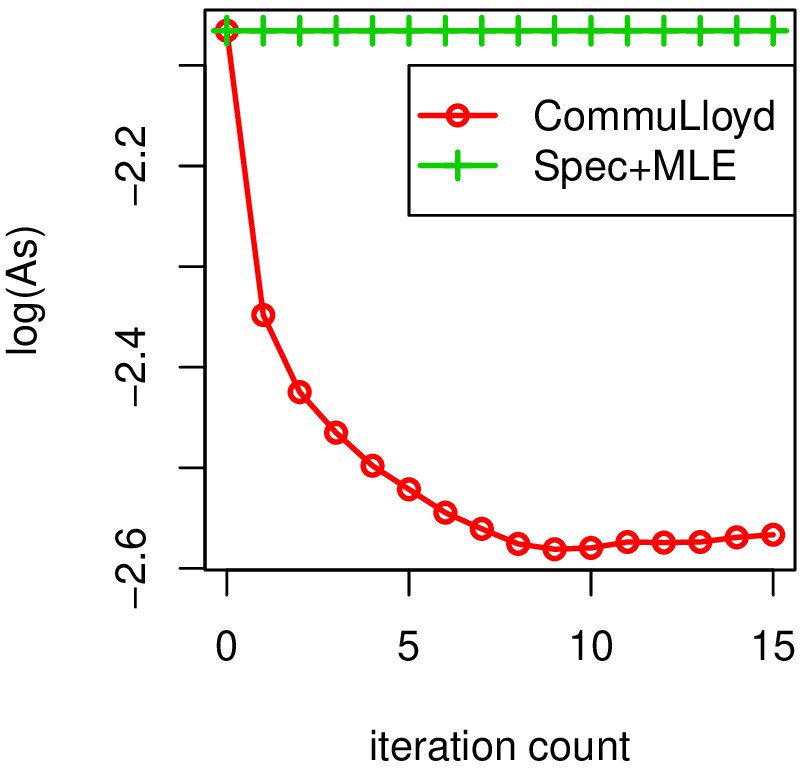}
\caption{}
\label{fig:SBMsparse}
\end{subfigure}%
\begin{subfigure}{.33\textwidth}
\centering
\includegraphics[width=1.0\linewidth]{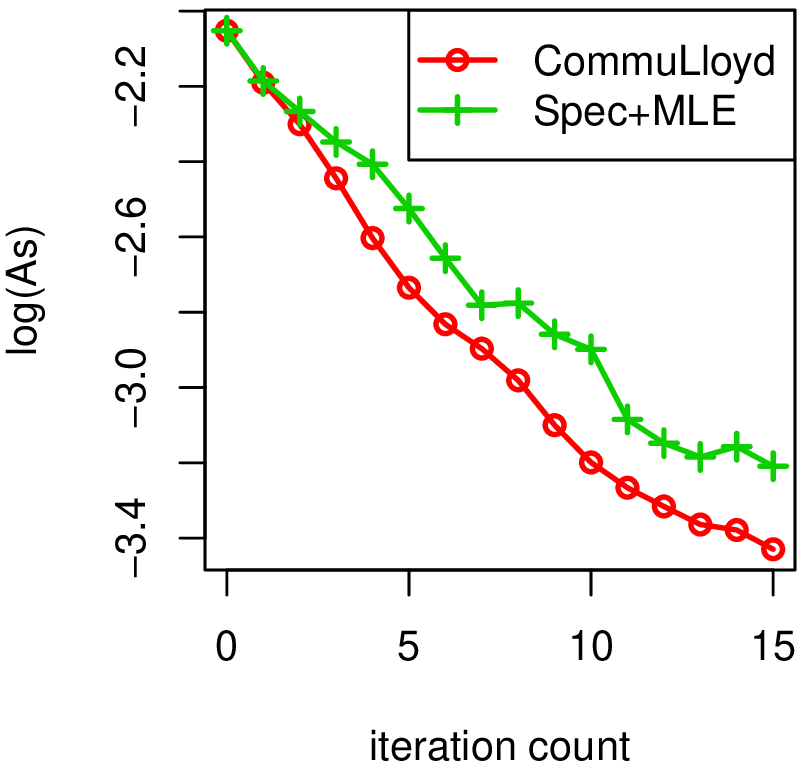}
\caption{}
\label{fig:SBMimba}
\end{subfigure}%
\caption{Plots of log mis-clustering rate $A_s$ versus the iteration.}
\label{figure:SBM}
\end{figure}


\subsubsection{Political Blog Data}
We now examine the performance of Algorithm 2 on a political blog dataset \cite{adamic2005political}. Each node in this data represents a web blog collected right after the 2004 US presidential election. Two nodes are connected if one contains a hyperlink to the other, and we ignore the direction of the hyperlink. After pre-processing \cite{karrer2011stochastic}, there are 1222 nodes and 16714 edges. Depending on whether liberal or conservative, the network naturally has two communities. We use the label given in \cite{adamic2005political} as our ground truth.

Table \ref{table:politic} reports the number of mis-labeled nodes of our algorithm on political blog dataset. Started from spectral clustering with an error rate of 437/1222, CommuLloyd achieves 56/1222 error rate after three iterations. After that, it fluctuates between 55 and 56 mis-labeled nodes. Our algorithm slightly outperforms the state-of-the-art algorithm SCORE \cite{jin2015fast}, which reports an error rate of 58/1222.

\begin{table} 
\begin{center}
\begin{tabular}{ c | c c c c c c c}
\hline
Methods & Spec & Spec+MLE(1st) & Spec+MLE(10th) & SCORE & CommuLloyd \\ \hline
No.mis-clustered & 437  & 132  & 61 & 58  & \textbf{56}\\ \hline
\end{tabular}
\end{center}
\caption{Number of mis-labeled blogs for different community detection methods. Here Spec is the spectral clustering method described in \cite{lei2013consistency}. We refer the one-step MLE refinement of Spec as Spec+MLE(1st), and 10 steps MLE refinements as Spec+MLE(10th). SCORE is a community detection algorithm for Degree-corrected SBM proposed in \cite{jin2015fast}. }  
\label{table:politic}
\end{table}

\subsection{Crowdsourcing} In this section, we compare CrowdLloyd algorithm with three other crowdsourcing methods: (1) Opt-D\&S, which is Spectral methods with one step EM update proposed in \cite{zhang2014spectral}, (2) MV, the majority voting estimator and (3) MV+EM, which is majority voting estimator followed by EM updates. Among those four algorithms, Opt-D\&S, MV and CrowdLloyd have provable error rates.

When there are ties, the majority voting estimator is not clearly defined. To avoid algorithmic randomness, we return the first label index when there are ties for CrowdLloyd and MV estimator. Since EM algorithm takes soft labels as input, we use weights rather than hard label of majority voting estimator in MV+EM. In practice, there are usually missing data. We denote by $0$ as the category of missing data and modify the label update step in CrowdLloyd as
\begin{equation} \label{eq:CSlabel2} 
\zsnext_j =  \argmin_{h \in [k]} \left[ \sum_{i, X_{ij} \neq 0} \sum_{g=1}^{k} \left( \mathbb{I}\{X_{ij}=g\} - \hat{\pi}_{igh}^{(s)}\right)^2. \right]
\end{equation}
\subsubsection{Simulation}
We follow the same setting as in the experiments of \cite{zhang2014spectral}, where there are $m=100$ workers and $n=1000$ items. Each item belongs to one of the $k=2$ categories. For worker $i$, the diagonal entries of his or her confusion matrix $\pi_{igg}$ are independently and uniformly drawn from the interval $[0.3, 0.9]$. We set $\pi_{igh} = 1- \pi_{igg}$ for all $i \in [n]$ and $(h,g) \in \{(1,2), (2,1)\}$. To simulate missing entries, we assume each entry $X_{ij}$ is observed independently with probability $p \in \{0.2,0.5,1\}$. 

We replicate experiments independently for 10 times and report the average results in Table \ref{table:CSsimu}. Among three provable algorithms, CrowdLloyd is consistently better than Opt D\&S (1st iteration) and MV under different missing levels. CrowdLloyd, MV+EM and Opt D\&S (10th iteration) have almost identical performance. As expected, CrowdLloyd is slightly worse because it does not use likelihood function to update the confusion matrices. As discussed before, an additional Neyman-Pearson testing step followed by CrowdLloyd algorithm would give the same output as EM updates. 

\begin{table} 
\begin{center}
\begin{tabular}{ | c | c | c | c |}
\hline
Methods & $p =1.0$ & $p=0.5$ & $p=0.2$ \\ \hline
Opt D\&S (1st iteration) & 0.06 &1.32 & 10.32  \\
Opt D\&S (10th iteration) & 0.07 & 1.10 & 8.22 \\
MV & 2.14 & 8.56 &18.97 \\
MV + EM & 0.07 &1.10 &8.17 \\
CrowdLloyd & 0.07 &1.14 &8.19\\
\hline
\end{tabular}
\end{center}
\caption{mis-clustering rate rate (\%) of different crowdsourcing algorithms on simulated data.}  
\label{table:CSsimu}
\end{table}

\subsubsection{Real Data} We now compare CrowdLloyd algorithm to other three algorithms on five real crowdsourcing data sets: (1) bird dataset is a binary task of labeling two bird species, (2)  recognizing textual entailment (RTE) dataset contains a set of sentence pairs and the task is to label whether the second sentence can be inferred from the first one, (3) TREC dataset is a binary task of assessing the quality of information retrieval in TREC 2011 crowdsourcing track, (4) dog dataset is a subset of Stanford dog dataset that requires labeling one of the four dog breeds, and (5) web dataset is a set of query-URL pairs for workers to label on a relevance scale from 1 to 5. 

The mis-clustering rates of four methods are summarized in Table \ref{table:CS}. CrowdLloyd slightly outperforms the other three on four of those five data sets. One possible explanation is that the least squares update $(\ref{eq:CSlabel})$ works for general sub-Gaussian distributions and thus is potentially more robust than the likelihood based EM update. 

\begin{table} 
\begin{center}
\begin{tabular}{ |c | c | c c c c |}
\hline
\#  & Date set &Opt D\&S & MV & MV+EM & CrowdLloyd  \\ \hline
1 & Bird & \textbf{10.09}   & 24.07  & 11.11 & 12.04 \\ 
2 & RTE  & 7.12  & 8.13 & 7.34 & \textbf{6.88} \\ 
3 & TREC  & 29.80  & 34.86 & 30.02 & \textbf{29.19} \\
4 & Dog & 16.89  & 19.58 & 16.66 & \textbf{15.99} \\
5 & Web & 15.86  & 26.93 & 15.74 & \textbf{14.25} \\ 
\hline
\end{tabular}
\end{center}
\caption{mis-clustering rate rate (\%) of different crowdsourcing algorithms on five real data sets}  \label{table:CS}
\end{table}

\section{Discussion} \label{sec:diss}
\subsection{Random Initialization}
In this section, we explore the condition under which random initialization converges to the global optimum.  Here we consider the symmetric two-mixture model (\ref{eq:modeltwo}). Note that Theorem \ref{thm:2main} holds for any initializer (could be data dependent) satisfying condition (\ref{eq:2cond}).  If we have a data independent initializer, condition (\ref{eq:2cond}) can be weakened 
and we have the following result on random initialization of the labels. 

\begin{thm} \label{thm:2improved}
Assume \[ \|\theta^*\| \ge C \sigma \left( \sqrt{\log(1/\delta)}+\left(d\log(1/\delta)(1+d/n)\right)^{1/4} \right) \] for a sufficiently large constant $C$. Suppose we independently draw $3 \log(1/\delta)$ random initializers. Then Lloyd's algorithm initialized by one of them can correctly classifies all the data points  with probability greater than $1-\delta$.
\end{thm}

The proof of Theorem \ref{thm:2improved} is given in appendix. Klusowski and Brinda \cite{klusowski2016statistical} showed similar results for EM algorithm under the condition that $\|\theta^*\| \gtrsim \sigma \sqrt{d\log d}$. We improve their condition to $\|\theta^*\| \gtrsim \sigma(d\log d)^{1/4}$ by choosing $\delta=1/d$. Recently, Xu et al. \cite{xu2016global} and  Daskalakis et al. \cite{daskalakis2016ten} proved the global convergence of EM algorithm on the population likelihood function. However, it is not clear what signal-to-noise ratio condition they need for the finite sample analysis. Moreover, our results can be generalized to $k$-mixtures and the sub-Gaussian case, but \cite{klusowski2016statistical}, \cite{xu2016global} and  \cite{daskalakis2016ten} all restrict their analyses to the two mixture Gaussian case.

As for the general $k$ mixture case, one popular way to initialize Lloyd's algorithm is to initialize $k$ centers by sampling without replacement from the data.  However, when $k \ge 3$, Jin et al. \cite{jin2016local} show even for well-separated ($\Delta \gtrsim \sqrt{d} \sigma$) spherical Gaussian mixtures, EM algorithm, initialized by above strategy, convergences to bad critical points with high probability. They prove it by constructing a bad case of two cluster of centers that are very far away with each other. We expect similar results for Lloyd's algorithm. 

\subsection{Convergence of the centers} As a byproduct of Theorem \ref{thm:kmeans}, we obtain the convergence rate of estimating Gaussian mixture's centers using Lloyd's algorithm. 
\begin{thm} \label{thm:centerslloyd}
Assume $n\alpha^2  > C k \log n  $ and $ r_k \ge C \sqrt{k} $ for a sufficiently large constant $C$. Then given any initializer satisfying
\begin{equation} \label{eq:disskcond}
G_0  < \left( \frac{1}{2} - \frac{6}{\sqrt{r_k}} \right)\frac{1}{\lambda} \quad \textrm{or} \quad \Lambda_0 \le \frac{1}{2} - \frac{4}{\sqrt{r_k}}
\end{equation}
with probability $1-\nu$, we have
\begin{equation} \label{eq:disscenter}
\max_{h \in [k]} \|\hat{\theta}_h^{(s)} - \theta_h\|^2 \le \frac{C_2 \sigma^2 (d + \log n)}{n} + \sigma^2 \exp \left( - \frac{\Delta^2}{24 \sigma^2}\right) \; \text{ for } s \ge 4\log n,
\end{equation}
for some constant $C_2$ with probability greater than $1-2\exp \left( - \frac{\Delta}{\sigma} \right) - \frac{4}{n}-\nu$.
\end{thm}
We show a linear convergence rate of estimating the centers. When the labels are known, the minimax optimal rate of estimating the centers under Euclidean distance is $\sigma \sqrt{d/n}$. Thus, we obtain the optimal rate under the condition that $n\alpha^2 \gtrsim k \log n$, $r_k \gtrsim\sqrt{k}$, $d \gtrsim \log n$ and $\Delta \gtrsim \sigma \log(n/d)$. 

\section{Proofs} \label{sec:proof}
This section is devoted to prove Theorem \ref{thm:2main}. The proofs of other theorems are collected in the appendix. 
\subsection{Preliminaries} \label{sec:2proofpre}
We rewrite the symmetric two-mixture model (\ref{eq:modeltwo}) as 
\begin{equation} \label{eq:model2}
z_i y_i = \theta^* + w_i, \quad i=1,2,\cdots,n,
\end{equation}
where $w_i = z_i \xi_i \sim \mathcal{N}(0,\sigma^2 I_d)$. To facilitate our proof, we present four lemmas on the concentration behavior of $w_i$. Their proofs are given in Section \ref{sec:techproof} of the appendix.

\begin{lemma} \label{lm:tech4}
For any $S \subseteq [n]$, define $W_S = \sum_{i \in S} w_i$. Then, 
\begin{equation} \label{eq:tech4}
\|W_S\| \le \sigma \sqrt{2(n+9d)|S|}  \;\; \text{ for all } S \subseteq [n]
\end{equation}
with probability greater than $1-\exp(-0.1n)$. 
\end{lemma}

\begin{lemma} \label{lm:tech1}
Let $\bar{w} = \frac{1}{n} \sum_{i=1}^{n} w_i$. Then we have
$\bar{w}' \theta^* \ge - \frac{\|\theta^*\|^2}{\sqrt{n}}$ and
$\|\bar{w}\|^2 \le \frac{3d\sigma^2}{n} + \frac{\|\theta^*\|^2}{n}$
with probability greater than $1 - 2 \exp \left(-\frac{\|\theta^*\|^2}{3\sigma^2}\right).$
\end{lemma}

\begin{lemma} \label{lm:tech6}
The maximum eigenvalue of $\sum_{i=1}^{n} w_iw_i'$ is upper bounded by $1.62(n+4d)\sigma^2$ with probability greater than $1-\exp(-0.1n)$. 
\end{lemma}

\begin{lemma} \label{lm:tech5}
For any fixed $i \in [n]$, $S \subseteq [n]$, $t>0$ and $\delta>0$,  we have
\[ \Prob \left\{ \iprod{w_i}{\frac{1}{|S|}\sum_{j \in S} w_j} \ge \frac{3\sigma^2(t\sqrt{|S|}+d+\log(1/\delta))}{|S|} \right \} \le \exp \left( - \min \left\{ \frac{t^2}{4d}, \frac{t}{4}\right\}  \right) + \delta.  \]
\end{lemma}

Let $\mathcal{E}$ be the intersection of high probability events in Lemma \ref{lm:tech4}, Lemma \ref{lm:tech1}, Lemma \ref{lm:tech6} and label initialization condition (\ref{eq:2cond}). Then $\mathbb{P}\{\mathcal{E}^c\} \le \nu+n^{-3}+\exp\left(-\|\theta^*\|^2/(3\sigma^2)\right)$. For the center initialization condition, we refer to the proof of Theorem \ref{thm:kmeans} in the appendix for more details. In the following, we deterministically analyze the error of estimating centers and labels on event $\mathcal{E}$. 
\noindent \paragraph{Error of the centers} We first decompose the error of estimating $\theta^*$ into two parts, where the first part comes from the error of estimating labels $z_i$, and the second part is due to the stochastic error $w_i$,
\begin{eqnarray*} \label{eq:decomp1}
\thetas  - \theta^* &=& \frac{1}{n} \sum_{i=1}^{n} \zs_i y_i  - \frac{1}{n} \sum_{i=1}^{n} z_i y_i + \frac{1}{n} \sum_{i=1}^{n} z_i y_i  - \theta^*
\end{eqnarray*}
Using (\ref{eq:model2}) and the fact that $\zs_i-z_i = -2\mathbb{I}\{\zs_i \neq z_i\} z_i $, we obtain
\begin{equation}   \label{eq:decomptheta}
\thetas  - \theta^* =  - \frac{2}{n} \sum_{i=1}^{n} \mathbb{I}\{\zs_i \neq z_i\} (\theta^*+w_i) + \frac{1}{n} \sum_{i=1}^{n} w_i = -2A_s \theta^* - 2R + \bar{w} 
\end{equation}
where $c_i = \mathbb{I}\{\zs_i \neq z_i\}$, $A_s = \frac{1}{n} \sum_{i=1}^{n} c_i$ and $R =  \frac{1}{n} \sum_{i=1}^{n} c_i w_i$. \\

\noindent \paragraph{Error of the labels}  Now we analyze the error of $\zsnext_i$. The iteration of $\zsnext_i$ is equivalent to
$$ \zsnext_i = \argmin_{ g \in \{-1,1\} }  \|gy_i  - \thetas \|_2^2 = \argmax_{ g \in \{-1,1\} } \iprod{gy_i}{\thetas} =\argmax_{ g \in \{-1,1\} } gz_i\iprod{\theta^*+w_i}{\thetas}, $$
which gives us
$\mathbb{I}\{\zsnext_i \neq z_i \} = \mathbb{I} \left\{ \iprod{\theta^*+w_i}{\thetas} \le 0 \right\}$. From (\ref{eq:decomptheta}),
\begin{eqnarray*}
\iprod{\theta^*+w_i}{\thetas} &=& \iprod{\theta^*+w_i}{(1-2A_s)\theta^* - 2R + \bar{w}} \\
&=& (1-2A_s)\|\theta^*\|^2 - (2R-\bar{w})'\theta^* + \iprod{w_i}{\theta^* - 2A_s \theta^* + 2R - \bar{w}}.
\end{eqnarray*}
Applying Lemma \ref{lm:tech4} with $S = \{ i\in [n], c_i=1\}$, we obtain
\[ \|R\| \le \frac{\sigma}{n} \sqrt{2(n+9d) nA_s} \le  \frac{\|\theta^*\|}{r} \sqrt{2A_s}, \]
which, together with Lemma \ref{lm:tech1}, implies
\[ (2R-\bar{w})'\theta^* \le 2\|R\|\|\theta^*\| - \bar{w}'\theta^* \le \frac{2\sqrt{2A_s}}{r} \|\theta^*\|^2 + \frac{1}{\sqrt{n}} \|\theta^*\|^2. \]
Consequently, on event $\mathcal{E}$, we have
\begin{equation}\label{eq:two.error.label}
\mathbb{I}\{\zsnext_i \neq z_i \} \le \mathbb{I} \left\{ \beta_0 \|\theta^*\|^2 + \iprod{w_i}{\theta^*} + \iprod{w_i}{- 2A_s \theta^* + 2R - \bar{w}} \le 0 \right\},
\end{equation}
where $\beta_0=1-2A_s-\frac{2\sqrt{2A_s}}{r} - \frac{1}{\sqrt{n}} \ge 1-2A_s-\frac{2}{r} - \frac{1}{\sqrt{n}}$. Now we are ready to prove Theorem \ref{thm:2main}. In the following, we are going to prove (\ref{eq:twoconst}) and (\ref{eq:tworate}) separately based on two different decompositions of the RHS of (\ref{eq:two.error.label}).

\subsection{Proof of (\ref{eq:twoconst}) in Theorem \ref{thm:2main}} \label{sec:2proof1}
Using the fact that for any $a, b \in \mathbb{R}$ and $c>0$,
$\mathbb{I}\{a+b \le 0\} \le \mathbb{I}\{a \le c\} + \mathbb{I}\{b \le -c\} \le \mathbb{I}\{a \le c\} + \frac{b^2}{c^2},$
we obtain
\begin{eqnarray*}
\mathbb{I}\{\zsnext_i \neq z_i \}
\le \mathbb{I} \left\{ \beta \|\theta^*\|^2 \le - \iprod{w_i}{\theta^*} \right\} + \frac{\left( w_i' (2R-\bar{w}-2A_s\theta^*) \right)^2}{\delta^2\|\theta^*\|^4},
\end{eqnarray*}
with $\beta = \beta_0 -\delta$ and  $\delta=3.12/r$. Taking an average of the equation above over $i \in [n]$, we have $A_{s+1} \le I_1 + I_2$, where
\begin{eqnarray} 
\label{eq:twoterm1} I_1 & = & \frac{1}{n} \sum_{i=1}^{n} \mathbb{I} \left\{ \iprod{w_i}{\theta^*} \le -\beta \|\theta^*\|^2 \right\},  \\
\label{eq:twoterm2} I_2 & = &  \frac{1}{n\delta^2 \|\theta^*\|^4} \sum_{i=1}^{n} \left( w_i' (2R-\bar{w}-2A_s\theta^*) \right)^2.
\end{eqnarray}

\noindent \paragraph{Upper Bound $I_1$} Note that $nA_s$ only takes discrete values $\{1,2,\cdots,\lfloor \frac{n}{2} \rfloor \}$. Let us define $\gamma_a = 1-2a-\frac{5.12}{r} - \frac{1}{\sqrt{n}}$ and 
$$T_i(a)= \mathbb{I} \left\{ \iprod{w_i}{\theta^*} \le - \gamma_a \|\theta^*\|^2 \right\}$$
for $a \in \mathcal{D}\triangleq \{\frac{1}{n},\frac{2}{n},\cdots, \frac{\lfloor \frac{n}{2} \rfloor}{n} \}$. For any fixed $a$, $\{T_i(a), i=1,2,\cdots,n\}$ are independent Bernoulli random variables, then Hoeffding's inequality implies
$$ \Prob\left\{ \frac{1}{n} \sum_{i=1}^{n} \left( T_i(a) - \E T_i(a) \right) \ge \sqrt{\frac{4\log (n/2)}{n}} \right\} \le \frac{1}{n^4}. $$ Thus by union bound, we have
\begin{equation} \label{eq:concenT} 
\Prob\left\{ \max_{a \in \mathcal{D}} \left[ \frac{1}{n} \sum_{i=1}^{n} \left( T_i(a) - \E T_i(a) \right) \right] \ge \sqrt{\frac{4 \log (n/2)}{n}}  \right\} \le \frac{1}{n^3}.
\end{equation}
Now it remains to upper bound the expectation of $T_i(a)$. Since $\iprod{w_i}{\theta^*}$ is sub-Gaussian with parameter $\sigma\|\theta^*\|$, Chernoff's bound yields
\begin{equation}\label{eq:expectT}
\E T_i(a) \le \exp \left( -\frac{  \gamma_a^2  \|\theta^*\|^2}{2\sigma^2} \right).
\end{equation}  
Consider the event
$$\mathcal{E}_1 = \left\{ \frac{1}{n} \sum_{i=1}^{n} T_i(a) \le \exp \left( -\frac{\gamma_a^2 \|\theta^*\|^2}{2\sigma^2} \right) + \sqrt{\frac{4 \log (n/2)}{n}}, \forall\; a \in \mathcal{D}  \right\}. $$
Equation (\ref{eq:concenT}) together with (\ref{eq:expectT}) implies $\Prob\{\mathcal{E}_1^c\} \le n^{-3}$. Therefore,
\begin{equation} \label{eq:twoI1bound}
I_1 \le \exp \left( -\frac{\gamma_{A_s}^2 \|\theta^*\|^2}{2\sigma^2} \right) + \sqrt{\frac{4 \log (n/2)}{n}}
\end{equation}
with probability greater than $1-n^{-3}$. \\

\noindent \paragraph{Upper Bound $I_2$}
Lemma \ref{lm:tech6} implies
\begin{eqnarray*}
I_2 &\le& \frac{\|2R-\bar{w}-2A_s\theta^*\|^4}{n\delta^2 \|\theta^*\|^4} \lambda_{max} \left(\sum_{i=1}^{n} w_iw_i' \right) \\
&\le&  \frac{1.62(1+4d/n)\sigma^2}{\delta^2 \|\theta^*\|^4} \|2R-\bar{w}-2A_s\theta^*\|^2,
\end{eqnarray*}
where $\lambda_{max}(B)$ stands for the maximum eigenvalue of the symmetric matrix $B$. By the Cauchy-Schwarz inequality, Lemma \ref{lm:tech4} and Lemma \ref{lm:tech1}, we have
\begin{eqnarray*} 
\|2R-\bar{w}-2A_s\theta^*\|^2 &\le& (1+1+4) \left(4\|R\|^2 + \|\bar{w}\|^2 + A_s^2\|\theta^*\|^2\right) \\
&\le& 6\|\theta^*\|^2 \left( \frac{8}{r^2}A_s + \frac{1}{r^2} + \frac{1}{n} + A_s^2 \right) .
\end{eqnarray*}
Since $\delta = 3.12/r$, we obtain
\begin{equation} \label{eq:twoI2bound}
I_2 \le \frac{8}{r^2}A_s + \frac{1}{r^2} + \frac{1}{n} + A_s^2.
\end{equation}
Recall that $A_{s+1} \le I_1 + I_2$. Combining (\ref{eq:twoI1bound}) and (\ref{eq:twoI2bound}), we have
\[ A_{s+1} \le \exp \left( -\frac{\gamma_{A_s}^2 \|\theta^*\|^2}{2\sigma^2} \right) + \sqrt{\frac{4\log n}{n}} + \frac{8}{r^2}A_s + \frac{1}{r^2} + A_s^2
 \]
on event $\mathcal{E} \cap \mathcal{E}_1$. Let $\tau = \frac{1}{2} - \frac{1}{\sqrt{n}} - \frac{2.56+\sqrt{\log r}}{r}$. From above inequality and the assumption that $A_{0} \le \tau,$ it can be proved by induction that $A_{s} \le \tau$ for all $s \ge 0$ when $r \ge 32$ and $n \ge 1600$. Consequently, we have $\gamma_{A_s} \ge \frac{2\sqrt{\log r}}{r}$ for all $s \ge 0$.  Plugging this bound of $\gamma_{A_s}$ into above inequality completes the proof.

\subsection{Proof of (\ref{eq:tworate}) in Theorem \ref{thm:2main}}
By (\ref{eq:twoconst}) in Theorem \ref{thm:2main} and the fact that $A_s \le \tau$ when $r \ge 32$, 
\[ A_{s+1} \le 0.5 A_s + \frac{2}{r^2} + \sqrt{\frac{4\log n}{n}}, \]
which implies $A_s \le 4r^{-2} + 5\sqrt {(\log n)/n}$ for all $s \ge \log n$. Then when $s \ge \log n$,  (\ref{eq:two.error.label}) implies
\begin{eqnarray*} 
&& \mathbb{I}\{\zsnext_i \neq z_i \} \\
&\le& \mathbb{I} \left\{ \left(1-\frac{3}{r}-\frac{1+10\sqrt{\log n}}{\sqrt{n}}\right)\|\theta^*\|^2 + \iprod{w_i}{\theta^*- 2A_s \theta^* + 2R - \bar{w}} \le 0 \right\}. 
\end{eqnarray*}
To get a better upper bound for mis-clustering rate, we further decompose $\iprod{w_i}{2R-2A_s\theta^*-\bar{w}}$ as $\iprod{w_i}{2R-2A_s\theta^*} -\iprod{w_i}{\bar{w}}$. Following similar arguments as in the proof of Theorem \ref{thm:2main}(a), we obtain $A_{s+1} \le J_1 + J_2 + J_3$, where $J_1$ is $I_1$ defined in (\ref{eq:twoterm1}) with $\beta =1 - 6.72/r - (3+10\sqrt{\log n})/\sqrt{n}$, $J_2$ and $J_3$ are defined as follows,
\begin{eqnarray}
\label{eq:twoJ2} J_2 & = &  \frac{r^2}{8.1 n \|\theta^*\|^4} \sum_{i=1}^{n} \left( w_i' (2R-2A_s\theta^*) \right)^2, \\
\label{eq:twoJ3} J_3 & = &  \frac{1}{n} \sum_{i=1}^{n} \mathbb{I} \left\{ \left( \frac{1}{2r} + \frac{2}{\sqrt{n}} \right) \|\theta^*\|^2 \le -\iprod{w_i}{\bar{w}} \right\}.
\end{eqnarray}
From (\ref{eq:expectT}), when $r \ge 32$ and $n \ge 40000$, $\E J_1 \le \exp(-\frac{\|\theta^*\|^2}{8\sigma^2})$. Analogous to (\ref{eq:twoI2bound}), we have 
\[J_2 \le \frac{8}{r^2} A_s + A_s^2 \le \left( \frac{8}{r^2} + \frac{8}{r^2} \right) A_s \le \frac{1}{r} A_s \]
on event $\mathcal{E}$. Choosing $S=[n]$, $t = \left(\sqrt{d} + \frac{\|\theta^*\|}{\sigma}\right) \frac{\|\theta^*\|}{\sigma}$ and $\delta=\exp \left( - \frac{\|\theta^*\|^2}{4\sigma^2} \right)$ in Lemma \ref{lm:tech5} yields
\[ \E J_3 \le 2 \exp \left( - \frac{\|\theta^*\|^2}{4\sigma^2} \right).  \] 
Combine the pieces and note that $A_{s+1} \le 1$,
\begin{eqnarray} 
\nonumber \E A_{s+1} &\le& \E \left [(J_1 + J_2 + J_3) \mathbb{I} \{ \mathcal{E} \} \right]+ \E \left[A_{s+1} \mathbb{I}\{\mathcal{E}^c\} \right] \\
\label{eq:twoexpbound} &\le& \frac{1}{r}~\E A_s +  3\exp \left(-\frac{\|\theta^*\|^2}{8\sigma^2}\right).
\end{eqnarray}
Then, by induction, we have $$\E A_s \le \frac{1}{r^{s-\lceil \log n \rceil}} + 4 \exp \left(-\frac{\|\theta^*\|^2}{8\sigma^2} \right),$$ when $r \ge 32$. Setting $s=\lceil 3 \log n \rceil$ and applying Markov's inequality, we obtain
\[ \mathbb{P} \left\{ A_s \ge t \right\} \le 4\exp \left(-\frac{\|\theta^*\|^2}{8\sigma^2} - \log t \right) + \frac{1}{n^2t}. \]
If $\frac{\|\theta^*\|^2}{8\sigma^2} \le 2 \log n$, a choice of $t = \exp \left( - \frac{\|\theta^*\|^2}{16\sigma^2} \right)$ gives us the desired result. Otherwise, since $A_s$ only takes discrete values $\{0, \frac{1}{n}, \frac{2}{n}, \cdots, 1\}$,  choosing $t = \frac{1}{n}$, we have
\[ \Prob \left\{ A_s > 0 \right \} =  \Prob \left\{ A_s \ge \frac{1}{n} \right \} \le 4n \exp(-2\log n) +  \frac{1}{n} \le \frac{5}{n}. \] 
Therefore, the proof of (\ref{eq:tworate}) is complete. When $r \to \infty$, we can slightly modify the proof to get the improved error exponent. For more details, we refer to the proof of Theorem \ref{thm:kmeans} in Section \ref{sec:kmixture}.


\appendix
\section{Proof of sub-Gaussian mixtures}\label{sec:kmixture}
\subsection{Preliminaries}
Let us begin by introducing some notaion. For any $S \subseteq [n]$, define $W_{S} = \sum_{i \in S} w_i$. Recall that $T_g^* = \left\{i \in [n], z_i=g \right\}$ and $T_g^{(s)} = \left\{i \in [n], \hat{z}_i^{(s)}=g \right\}$, let us define 
\[ S_{gh}^{(s)} = \left\{ i \in [n], z_i=g, \zs_i=h\right\} = T_g^* \cap T_h^{(s)}.\]
Then we have $n_h^{(s)} = \sum_{g \in [k]} n_{gh}^{(s)}$ and $n_h^*= \sum_{g \in [k]} n_{hg}^{(s)}$. In the rest of the proof, we will sometimes drop the upper index $(s)$ of $n_{gh}^{(s)}$, $n_h^{(s)}$ and $S_{gh}^{(s)}$ when there is no ambiguity. Also, we suppress the dependence of $k$ by writing $r_k$ as $r$.

Analogous to Lemma \ref{lm:tech4} - Lemma \ref{lm:tech5}, we have following technical lemmas on the concentration behavior of sub-Gaussian random vectors $\{w_i\}$.

\begin{lemma} \label{lm:tech1.1}
\begin{equation} \label{eq:tech1.1}
\|W_{S}\| \le \sigma \sqrt{3(n+d)|S|}  \;\; \text{ for all } S \subseteq [n].
\end{equation}
with probability greater than $1-\exp(-0.3n)$.
\end{lemma}
\begin{lemma} \label{lm:tech1.2}
\begin{equation} \label{eq:tech1.2}
\lambda_{max} \left(\sum_{i=1}^{n} w_i w_i' \right) \le 6\sigma^2(n+d).
\end{equation}
with probability greater than $1-\exp(-0.5n)$.
\end{lemma}

\begin{lemma} \label{lm:tech1.3}
For any fixed $i \in [n]$, $S \subseteq [n]$, $t>0$ and $\delta>0$,
\[ \Prob \left\{ \iprod{w_i}{\frac{1}{|S|}\sum_{j \in S} w_j} \ge \frac{3 \sigma^2(t\sqrt{|S|} + d+\log(1/\delta))}{|S|} \right \} \le \exp \left( - \min \left\{ \frac{t^2}{4d}, \frac{t}{4}\right\}  \right) + \delta.  \]
\end{lemma}

\begin{lemma} \label{lm:tech3} 
\begin{equation} \label{eq:tech3}
\|W_{T_h^*}\| \le 3 \sigma \sqrt{(d+\log n)|T_h^*|}  \;\; \text{ for all } h \in [k]
\end{equation}
with probability greater than $1-n^{-3}$.
\end{lemma}

\begin{lemma} \label{lm:tech2}
For any fixed $\theta_1, \cdots, \theta_k \in \mathbb{R}^d$ and $a>0$, we have
\begin{equation}\label{eq:tech2}
\sum_{i \in T_g^*} \mathbb{I}\left\{ a\|\theta_h-\theta_g\|^2 \le \iprod{w_i}{\|\theta_h-\theta_g\|} \right\}
\le n_g^* \exp \left( - \frac{a^2 \Delta^2}{2\sigma^2} \right) + \sqrt{5 n_g^* \log n}
\end{equation}
for all $g \neq h \in [k]^2$ with probability greater than $1-n^{-3}$.
\end{lemma}

\subsection{Two Key Lemmas}
The following two lemmas give the iterative relationship between the error of estimating centers and the error of estimating labels. Let $\mathcal{E}$ be the intersection of high probability events in Lemma \ref{lm:tech1.1}, Lemma \ref{lm:tech1.2} Lemma \ref{lm:tech3}, Lemma \ref{lm:tech2} and the initialization condition (\ref{eq:kcond}).  Then we have $\mathbb{P}\{\mathcal{E}^c\} \le 3n^{-3}+\nu$. In the rest part of the proof, if not otherwise stated, we all condition on the event $\mathcal{E}$ and the following analysis are deterministic. 

\begin{lemma} \label{lm:center}
On event $\mathcal{E}$, if $G_s \le \frac{1}{2}$, then we have
\begin{equation} \label{eq:centerineq}
\Lambda_s \le  \frac{3}{r} + \min \left\{ \frac{3}{r} \sqrt{kG_s} + 2G_s \Lambda_{s-1},  \lambda G_s \right\}.
\end{equation}
\end{lemma}

\begin{lemma} \label{lm:labelupdate}
On event $\mathcal{E}$, if $\Lambda_s \le \frac{1-\epsilon}{2}$ and $r \ge 36 \epsilon^{-2}$, then
\begin{equation} \label{eq:labelineq}
G_{s+1} \le \frac{2}{\epsilon^4 r^2} + \left(\frac{28}{\epsilon^2 r} \Lambda_s\right)^2 + \sqrt{\frac{5k \log n}{\alpha^2 n}}.
\end{equation}
\end{lemma}

\begin{proof}[Proof of Lemma \ref{lm:center}]
For any $B \subseteq [n]$, define $\bar{Y}_{B} = \frac{1}{|B|} \sum_{i \in B} y_i$. The error of estimated centers at step $s$ can be written as
\begin{eqnarray*}
\thetas_h - \theta_h &=& \frac{1}{n_h} \sum_{i \in S_{hh}} (y_i - \theta_h) + \frac{1}{n_h} \sum_{a \neq h} \sum_{i \in S_{ah}} (y_i - \theta_h) \\
&=& \frac{1}{n_h} \sum_{i \in S_{hh}} w_i + \sum_{a \neq h} \frac{n_{ah}}{n_h} \left( \bar{Y}_{S_{ah}} - \theta_h \right) 
\end{eqnarray*}
According to our label update step, we have $\|y_i - \hat{\theta}^{(s-1)}_h\| \le \|y_i - \hat{\theta}^{(s-1)}_a\|$ for any $i \in S_{ah}$. This means for any $i \in S_{ah}$, $y_i$ is closer to $\hat{\theta}^{(s-1)}_h$ than $\hat{\theta}^{(s-1)}_a$, so is the average of $\{y_i, i \in S_{ah}\}$. Thus, we have
$$\| \bar{Y}_{S_{ah}}- \hat{\theta}^{(s-1)}_h \| \le  \| \bar{Y}_{S_{ah}}- \hat{\theta}^{(s-1)}_a \|.$$  Consequently, triangle inequality gives us
\[ \left\| \bar{Y}_{S_{ah}} - \theta_h \right\| \le \left\| \bar{Y}_{S_{ah}} - \theta_a \right \| + \| \hat{\theta}^{(s-1)}_a - \theta_a \| +  \| \hat{\theta}^{(s-1)}_h - \theta_h \|,   \]
which, combined with Lemma \ref{lm:tech1.1} and the definition of $\Lambda_{s-1}$, yields
\[  \left\| \bar{Y}_{S_{ah}} - \theta_h \right\|  \le \sigma \sqrt{3(n+d)/n_{ah}} + 2 \Lambda_{s-1} \Delta.  \]
Taking a weighted sum over $a \neq h \in [k]$, we get 
\begin{eqnarray*}
\sum_{a \neq h} \frac{n_{ah}}{n_h} \left\| \bar{Y}_{S_{ah}} - \theta_h \right \|
&\le& \frac{\sigma \sqrt{3(n+d)}}{n_h} \sum_{a \neq h} \sqrt{ n_{ah}} + 2\Lambda_{s-1} \Delta \sum_{a \neq h} \frac{n_{ah}}{n_h} \\
&\le& \frac{\sigma \sqrt{3(n+d)}}{ \sqrt{n_h}} \sqrt{(k-1) G_s} + 2 G_s \Lambda_{s-1} \Delta,
\end{eqnarray*}
where the Last inequality is due to Cauchy-Schwartz and the fact that $\sum_{a \neq h} n_{ah} \le G_s n_h $. Note that $W_{S_{hh}} = W_{T_h^*} - \sum_{a \neq h} W_{S_{ha}}$. Triangle inequality and Lemma \ref{lm:tech2} imply
\begin{eqnarray*}
\left \| W_{S_{hh}} \right\| 
&\le& 3\sigma \sqrt{d + \log n} \sqrt{n_h^*} + \sigma \sqrt{3(n+d)} \sqrt{n_h^* - n_{hh}}.
\end{eqnarray*}
Since $G_s \le \frac{1}{2}$, we have
\begin{equation} \label{eq:boundnh}
n_h \ge n_{hh} \ge n_h^*(1-G_s) \ge \frac{1}{2} n_h^* \ge \frac{1}{2} \alpha n.
\end{equation}
Combining the pieces, we obtain
\begin{eqnarray}
\nonumber \left \| \thetas_h - \theta_h \right\| 
&\le& 3\sigma \sqrt{\frac{d+\log n}{\alpha n}} + 3 \sigma \sqrt{\frac{ k (n + d)}{\alpha n} G_s} + 2G_s \Lambda_{s-1} \Delta\\
\label{eq:centereq1} &\le& \left( \frac{3}{r} (1+\sqrt{kG_s}) + 2G_s \Lambda_{s-1} \right) \Delta.
\end{eqnarray}
Therefore, we get the first term in (\ref{eq:centerineq}). To prove the second term, we decompose $\thetas_h$ differently.
\begin{eqnarray}
\nonumber \thetas_h &=& \frac{1}{n_h} \sum_{i=1}^{n} \left( \theta_{z_i} + w_i \right) \mathbb{I} \left\{ \zs_i = h \right\} \\
\nonumber &=& \frac{1}{n_h} \sum_{a=1}^{k} \sum_{i=1}^{n} \theta_a \mathbb{I} \left\{ z_i=a, \zs_i = h \right\} + \frac{1}{n_h} \sum_{i \in T_h} w_i \\
\label{eq:thetah} &=& \sum_{a=1}^{k} \frac{n_{ah}}{n_h}  \theta_a  + \frac{1}{n_h} W_{T_h}.
\end{eqnarray}
Then, the error of $\thetas_h$ can be upper bounded as  
\begin{eqnarray*}
\left\| \thetas_h - \theta_h \right\| &=& \left\|  \sum_{a=1}^{k} \frac{n_{ah}}{n_h}  (\theta_a - \theta_h) + \frac{1}{n_h} W_{T_h} \right\|
\le  \left\|  \sum_{a \neq h} \frac{n_{ah}}{n_h}  (\theta_a - \theta_h) \right\| +  \left\| \frac{1}{n_h} W_{T_h} \right\|.
\end{eqnarray*}
By triangle inequality,
\begin{equation} \label{eq:aaa} 
\left\|  \sum_{a \neq h} \frac{n_{ah}}{n_h}  (\theta_a - \theta_h) \right\| \le \sum_{a \neq h} \frac{n_{ah}}{n_h} \left\|\theta_a - \theta_h\right\| \le \lambda \Delta \sum_{a \neq h} \frac{n_{ah}}{n_h} \le \lambda \Delta G_s.
\end{equation}
This, together with Lemma \ref{lm:tech1.1} and (\ref{eq:boundnh}), implies
\begin{equation} \label{eq:errorcenter}
\left\| \thetas_h - \theta_h \right\|  \le  \lambda \Delta G_s + \sigma \sqrt{\frac{3(n+d)}{n_h}} \le \left( \lambda G_s  + \frac{3}{r} \right) \Delta
\end{equation}
for all $h \in [k]$. The proof is complete.\\
\end{proof}

\begin{proof}[Proof of Lemma \ref{lm:labelupdate}]
For any $g \neq h \in [k] \times [k]$,
\begin{eqnarray} 
\nonumber \mathbb{I} \left\{ z_i = g, \zsnext_i= h \right\} &\le& \mathbb{I} \left\{ \| \theta_g + w_i - \thetas_h \|^2 \le \|\theta_g + w_i - \thetas_g\|^2 \right\} \\
\label{eq:errorgh} &=& \mathbb{I} \left\{ \| \theta_g - \thetas_h \|^2 - \|\theta_g - \thetas_g\|^2 \le 2 \iprod{w_i}{\thetas_h - \thetas_g} \right\}.
\end{eqnarray}
Triangle inequality implies
\begin{equation*}
\|\theta_g-\thetas_h\|^2 
\ge \left(\|\theta_g-\theta_h\| - \| \theta_h - \thetas_h \| \right)^2
\ge \left( 1 - \Lambda_s  \right)^2 \|\theta_g-\theta_h\|^2.
\end{equation*}
Using the fact that $(1-x)^2-y^2 \ge (1-x-y)^2$ when $y(1-x-y) \ge 0$, we obtain
\begin{equation}
\| \theta_g - \thetas_h \|^2 - \|\theta_g - \thetas_g\|^2 = \left( 1 - 2 \Lambda_s \right)^2 \|\theta_g-\theta_h\|^2 
 \label{eq:LHStesting} \ge \epsilon^2 \|\theta_g-\theta_h\|^2.
\end{equation}
Denote by $\Delta_h = \hat{\theta}_h^{(s)} - \theta_h$ for $h \in [k]$. Then,
\begin{eqnarray} 
\nonumber && \mathbb{I} \left\{ z_i = g, \zsnext_i= h \right\}  \\
\nonumber &\le& \mathbb{I} \left\{ \epsilon^2 \|\theta_g - \theta_h\|^2 \le 2 \iprod{w_i}{\theta_h-\theta_g+ \Delta_h - \Delta_g} \right\} \\
\nonumber \label{eq:errordecomp} &\le& \mathbb{I} \left\{\frac{\epsilon^2}{2}\|\theta_g - \theta_h\|^2 \le 2 \iprod{w_i}{\theta_h-\theta_g}\right\} +  \mathbb{I} \left\{ \frac{\epsilon^2}{2}\Delta^2 \le 2 \iprod{w_i}{\Delta_h - \Delta_g}\right\}.
\end{eqnarray}
Taking a sum over $i \in T_g^*$ and using Markov's inequality on the second term, we obtain
\begin{equation}
\label{eq:nghbound} n_{gh}^{(s+1)} \le \sum_{i \in T_g^*}  \mathbb{I} \left\{\frac{\epsilon^2}{4}\|\theta_g - \theta_h\|^2 \le \iprod{w_i}{\theta_h-\theta_g}\right\}
+ \sum_{i \in T_g^*} \frac{16}{\epsilon^4 \Delta^4} \left( w_i'(\Delta_h - \Delta_g) \right)^2
\end{equation}
Note that $\mathbb{I} \left\{ \frac{\epsilon^2}{4} \|\theta_g - \theta_h\|^2 \le \iprod{w_i}{\theta_h-\theta_g}\right\}$ are independent Bernoulli random variables. By Lemma \ref{lm:tech2}, the first term in RHS of (\ref{eq:nghbound}) can be upper bounded by
\begin{equation} \label{eq:l1} 
n_g^* \exp \left( - \frac{\epsilon^4 \Delta^2}{32\sigma^2} \right) + \sqrt{5 n_g^* \log n}.
\end{equation}
By Lemma \ref{lm:tech1.2}, the second term in RHS of (\ref{eq:nghbound}) can be upper bounded by
\begin{equation} \label{eq:l2}
\sum_{i \in T_g^*} \frac{16}{\epsilon^4 \Delta^4} \left( w_i'(\Delta_h - \Delta_g) \right)^2 \le \frac{96 (n_g^*+d)
\sigma^2}{\epsilon^4 \Delta^4} \|\Delta_g - \Delta_h\|^2. 
\end{equation}
Combining (\ref{eq:nghbound}), (\ref{eq:l1}) and (\ref{eq:l2}) and using the fact that $\|\Delta_g - \Delta_h\|^2 \le 4 \Lambda_s^2 \Delta^2$, we get
\begin{equation*}
n_{gh}^{(s+1)} \le n_g^* \exp \left( - \frac{\epsilon^4 \Delta^2}{32\sigma^2} \right) + \sqrt{5n_g^* \log n} + \frac{384 (n_g^*+d)\sigma^2}{\epsilon^4 \Delta^2} \Lambda_s^2.
\end{equation*}
Consequently, 
\begin{equation} \label{eq:Gsterm1} 
\max_{g \in [k]}  \sum_{h \neq g} \frac{n_{gh}^{(s+1)}}{n_g^*}  \le k \exp \left( - \frac{\epsilon^4 \Delta^2}{32\sigma^2} \right)+ k \sqrt{\frac{5 \log n}{\alpha n}} + \frac{384}{\epsilon^4 r^2 } \Lambda_s^2.
\end{equation}
Since $\Lambda_s \le 1/2$ and $r \ge 20 \epsilon^{-2}$, the RHS of (\ref{eq:Gsterm1}) is smaller that $1/2$ when $ \alpha n \ge 32 k^2 \log n$. Thus, 
$$n_h^{(s+1)} \ge n_{hh}^{(s+1)} \ge \frac{1}{2}n_h^* \ge \frac{1}{2}\alpha n $$
for all $h \in [k]$ and we have
\begin{equation} \label{eq:Gsterm2}  
\max_{h \in [k]} \sum_{g \neq h} \frac{n_{gh}^{(s+1)}}{n_h^{(s+1)}}   \le \frac{2}{\alpha} \exp \left( - \frac{\epsilon^4 \Delta^2}{32\sigma^2} \right) + \sqrt{\frac{5 k \log n}{\alpha^2 n}} + \frac{768}{\epsilon^4 r^2} \Lambda_s^2, 
\end{equation}
which, together with (\ref{eq:Gsterm1}), implies
\[ G_{s+1} \le \exp \left( - \frac{\epsilon^4 \Delta^2}{32\sigma^2} + \log(2/\alpha) \right) + \sqrt{\frac{5k \log n}{\alpha^2 n}} + \frac{768}{\epsilon^4 r^2} \Lambda_s^2  \]
Under the assumptions that $\epsilon^4 \alpha \Delta^2/\sigma^2 \ge r^2 \epsilon^4 \ge 36$, we have the desired result (\ref{eq:labelineq}).
\end{proof}

\subsection{Proof of Theorem \ref{thm:kmeans}}
\begin{proof} [Proof of (\ref{eq:kconst}) in Theorem \ref{thm:kmeans}]
From Lemma \ref{lm:center}, a necessary condition for $\Lambda_0 \le \frac{1}{2} - \frac{4}{\sqrt{r}}$ is $G_0 \le (\frac{1}{2} - \frac{6}{\sqrt{r}})\frac{1}{\lambda}$. Setting $\epsilon=\frac{7}{\sqrt{r}}$ in Lemma \ref{lm:labelupdate}, we have $G_1 \le 0.35$. Plugging it into Lemma \ref{lm:center} gives us $\Lambda_{1} \le 0.4$, under the assumption that $r \ge 16\sqrt{k}$. Then it can be easily proved by induction that $G_{s} \le 0.35$ and $\Lambda_s \le 0.4$ for all $s \ge 1$. Consequently, Lemma \ref{lm:center} yields
\[ \Lambda_{s} \le \frac{3}{r} + \frac{3}{r} \sqrt{kG_s} + G_s \le \frac{1}{2} + G_s\]
which, combined with (\ref{eq:labelineq}), implies
\[ G_{s+1} \le \frac{C}{r^2} + \frac{C}{r^2} \left(\frac{1}{4}+2G_s+G_s^2 \right) + \sqrt{\frac{5k \log n}{\alpha^2 n}} \le \frac{2C}{r^2} + \frac{3C}{r^2} G_s  +  \sqrt{\frac{5k \log n}{\alpha^2 n}} \]
for some constant $C$. Here we have chosen $\epsilon=1/5$ in Lemma \ref{eq:labelineq} to get the first inequality. \end{proof}

\begin{proof}[Proof of (\ref{eq:efficiency}) in Theorem \ref{thm:kmeans}]
From the proof of Lemma \ref{lm:center}, the error of estimating $\theta_h$ at iteration $s$ can be written as $\hat{\theta}_h^{(s)} - \theta_h = \frac{1}{n_h} W_{T_h^*} + u_h$, with
\begin{equation} \label{eq:uhnorm} 
\|u_h\| \le \left( \frac{3}{r} \sqrt{kG_s} + G_s \right) \Delta \le \sqrt{G_s} \Delta
\end{equation}
In addition, by Lemma \ref{lm:center} and Lemma \ref{lm:labelupdate}, there is a constant $C_1$ such that
\[ \Lambda_{s} \le \frac{3}{r} + \sqrt{G_s} + 2G_s \Lambda_{s-1} \le \frac{C_1}{r} + \frac{C_1}{r} \Lambda_{s-1} + 0.7 \Lambda_{s-1} + \left(\frac{C_1k \log n}{ \alpha^2 n} \right)^{1/4} \]
for all $s \ge 1$. Therefore, when $r$ is large enough, we have $$\Lambda \le C_2 r^{-1} + C_2 \left(\frac{k \log n}{ \alpha^2 n} \right)^{1/4}$$ for all $s \ge \log n$. Then by (\ref{eq:LHStesting}), we have
\begin{equation*}
\mathbb{I} \left\{ z_i = g, \zsnext_i= h \right\} \le \mathbb{I} \left\{ \beta_1 \|\theta_g - \theta_h\|^2 \le 2 \iprod{w_i}{\theta_h-\theta_g+\Delta_h - \Delta_g} \right\}
\end{equation*}
where $\left(1 -  2 \Lambda_s \right)^2 \ge \beta_1 :=  1 - 4C_2 r^{-1} - 4C_2\left(\frac{k \log n}{ \alpha^2 n} \right)^{1/4}.$ \\

In order to prove that $A_{s}$ attains convergence rates (\ref{eq:efficiency}), we first upper bound the expectation of $A_s$ and then derive the high probability bound using Markov's inequality. Similar to the two-mixture case, we need to upper bound the inner product $\iprod{w_i}{\Delta_h - \Delta_g}$ more carefully.  Note that $\{T_h^*, h \in [k]\}$ are deterministic sets, we could use concentration equalities to upper bound $W_{T_h^*}$ and $u_h$ parts separately.

Let $v_{h} = \frac{1}{n_h} W_{T_h^*}$ for $h \in [k]$ and we decompose $\mathbb{I} \left\{ z_i = g, \zsnext_i= h \right\} $ into three terms. 
\begin{eqnarray*} 
\mathbb{I} \left\{ z_i = g, \zsnext_i= h \right\} &\le& \mathbb{I} \left\{ \beta \|\theta_g - \theta_h\|^2 \le 2 \iprod{w_i}{\theta_h-\theta_g}\right\} \\
&&+  \mathbb{I} \left\{ \beta_2 \Delta^2 \le 2 \iprod{w_i}{u_h - u_g}\right\} \\
&& + \mathbb{I} \left\{ \beta_4 \Delta^2 \le 2 \iprod{w_i}{v_h - v_g} \right\},
\end{eqnarray*}
where $\beta_2$ and $\beta_4$ will be specified later and $\beta = \beta_1 - \beta_2 - \beta_4$. Taking a sum over $h \in [k]$ and $i \in [n]$, we obtain \[ \E A_{s+1} \le  \E J_1 + \E J_2 + \E J_3 \] with
\begin{eqnarray}
\label{eq:J1} J_1 &=& \sum_{h \in [k]}  \frac{1}{n} \sum_{i=1}^{n} \mathbb{I} \left\{ \beta \|\theta_{z_i} - \theta_h\|^2 \le 2 \iprod{w_i}{\theta_h-\theta_{z_i}}\right\} \\
\label{eq:J2} J_2 &=& \sum_{h \in [k]}  \frac{1}{n} \sum_{i=1}^{n}  \mathbb{I} \left\{ \beta_2 \Delta^2 \le 2 \iprod{w_i}{u_h - u_{z_i}}\right\}.\\
\label{eq:J3} J_3 &=& \sum_{h \in [k]} \frac{1}{n} \sum_{i=1}^{n} \mathbb{I} \left\{ \beta_4 \Delta^2 \le 2 \iprod{w_i}{v_{z_i} - v_h}\right\}.
\end{eqnarray}

Let us first consider the expectation of $J_1$. Using Chernoff's bound, we have
\[ \Prob  \left\{ \beta \|\theta_g - \theta_h\|^2 \le 2 \iprod{w_i}{\theta_h-\theta_g}\right\} \le \exp \left( - \frac{\beta^2 \|\theta_h-\theta_g\|^2}{8\sigma^2} \right) \le \exp \left( - \frac{\beta^2 \Delta^2}{8\sigma^2} \right). \]
Thus,
\[ \E J_1 \le k \exp \left( - \frac{\beta^2 \Delta^2}{8\sigma^2} \right) = \exp \left( - \frac{\gamma \Delta^2}{8\sigma^2} \right), \]
with $\gamma = \beta^2 - \frac{8\sigma^2 \log k}{\Delta^2} \ge \beta^2 - 8/r^2$. \\

We use Markov Inequality to upper bound $J_2$. Markov's inequality and Lemma \ref{lm:tech1.2} give us
\begin{eqnarray*}
\frac{1}{n} \sum_{i=1}^{n} \mathbb{I} \left\{ \beta_2 \Delta^2 \le 2 \iprod{w_i}{u_h - u_{z_i}}\right\} 
&\le& \frac{4}{n \beta_2^2 \Delta^4} \sum_{g \in [k]} \sum_{i \in T_g^*} \left( w_i' (u_h - u_{g})\right)^2 \\
&\le& \frac{24\sigma^2}{n \beta_2^2 \Delta^4} \sum_{g \in [k]} (n_g^* +d) \|u_h - u_{g}\|^2. 
\end{eqnarray*}
(\ref{eq:uhnorm}) implies
\begin{eqnarray*}
J_2 \le \frac{96\sigma^2 G_s }{n \beta_2^2 \Delta^2} \sum_{h \in [k]} \sum_{g \in [k]} (n_g^* + d)
\le \frac{96\sigma^2 k (n+kd)}{ \alpha n \beta_2^2 \Delta^2} A_s = \frac{12\sqrt{k}}{r} A_s.
\end{eqnarray*}
Here the second inequality is due to the fact that $G_s \le A_s/ \alpha$. And we choose $\beta_2 = \sqrt{8k/r}$ in the last equality. \\


\noindent Finally, we upper bound the expectation of $J_3$. Given $z_i=g$, we have
\begin{eqnarray*}
&& \Prob \left\{ \beta_4 \Delta^2 \le 2 \iprod{w_i}{v_g - v_h}\right\} \\
&\le& \Prob \left\{ \frac{\beta_4}{4} \Delta^2 \le  \iprod{w_i}{v_g}\right\} + \Prob \left\{ - \frac{\beta_4}{4} \Delta^2 \ge  \iprod{w_i}{v_h}\right\} \\
&\le& \Prob \left\{ \frac{\beta_4}{8} \Delta^2 \le  \iprod{w_i}{\frac{1}{n_g^*}W_{T_g^*}}\right\} + \Prob \left\{ - \frac{\beta_4}{8} \Delta^2 \ge  \iprod{w_i}{\frac{1}{n_h^*}W_{T_h^*}}\right\} \
\end{eqnarray*}
Choosing $t=\max\{\frac{\sqrt{d} \Delta}{\sigma},  \frac{\Delta^2}{\sigma^2}\}$, $\delta = \exp\left(-\frac{\Delta^2}{4\sigma^2}\right)$ in Lemma \ref{lm:tech1.3}, and $$\beta_4 = \frac{64}{r} \ge \frac{8}{\Delta^2} \left( \frac{3 \max\{\sqrt{d} \sigma \Delta, \Delta^2\}}{\sqrt{\alpha n}} + \frac{3\sigma^2 d + \Delta^2}{\alpha n}  \right),$$
we obtain $\Prob \left\{ \beta_4 \Delta^2 \le 2 \iprod{w_i}{v_g - v_h}\right\} \le 2\exp(-\Delta^2/(4\sigma^2))$, where we have used the assumption that $n_g^* \ge \alpha n$ and $\alpha n \ge 36 r^2$. Thus, \[ \E J_3 \le 2 k \exp \left( - \frac{\Delta^2}{\sigma^2} \right), \]

\noindent Combining the pieces,  we have
\begin{eqnarray*} 
\E A_{s+1} &\le& \E \left[ J_1 \right] + \E \left[ J_2 \mathbb{I}\{ \mathcal{E} \} \right] + \E \left[ J_3 \right]  + \mathbb{P}\{\mathcal{E}^c\} \\
&\le&  \exp \left( - \frac{\gamma \Delta^2}{8\sigma^2} \right) + \frac{12\sqrt{k}}{r} \E A_s + 2 k \exp \left( - \frac{\Delta^2}{\sigma^2} \right), 
\end{eqnarray*}
with $\gamma = (\beta_1 - \sqrt{8k/r} - 64/r)^2 - 8/r^2 = 1 - o(1)$. Here only prove the case that $r \to \infty$. For the finite case, all the $o(1)$ in the following proof can be substituted by a small constant. 
\[ \E A_s \le \frac{1}{2^{s-\lceil \log r \rceil}} + 2 \exp \left( -(1-\eta) \frac{\Delta^2}{8\sigma^2} \right) + \frac{2}{n^3}
\le 2 \exp \left( -(1-\eta) \frac{\Delta^2}{8\sigma^2} \right) + \frac{3}{n^3} \]
when $s \ge 4 \log n$. By Markov's inequality, for any $t>0$,
\begin{equation} \label{eq:mrkv} 
\Prob \left\{ A_s \ge t \right\} \le \frac{1}{t} \E A_s \le \frac{2}{t} \exp \left( -(1-\eta) \frac{\Delta^2}{8\sigma^2} \right) + \frac{3}{n^3 t}. 
\end{equation}
If $(1-\eta)\frac{\Delta^2}{8\sigma^2} \le 2\log n$, choose $t = \exp \left( -(1-\eta-\frac{8\sigma}{\Delta}) \frac{\Delta^2}{8\sigma^2} \right)$ and we have
\[ \Prob \left\{ A_s \ge   \exp \left( -(1-\eta-\frac{8\sigma}{\Delta}) \frac{\Delta^2}{8\sigma^2} \right) \right\} \le \frac{4}{n} + 2 \exp\left( - \frac{\Delta}{\sigma} \right). \]
Otherwise, since $A_s$ only takes discrete values of $\{0, \frac{1}{n}, \cdots, 1\}$, choosing $t=\frac{1}{n}$ in (\ref{eq:mrkv}) leads to
\[ \Prob \left\{ A_s > 0 \right \} =  \Prob \left\{ A_s \ge \frac{1}{n} \right \} \le 2n \exp(-2\log n) +  \frac{3}{n^2} \le \frac{4}{n}. \] 
The proof is complete. 
\end{proof}

\subsection{Proof of Lower Bounds}
The key difficulty in proving the lower bound is to deal with infinitum of all label permutations. Here we adapt the proof idea from \cite{gao2016community}. We define a subset of the parameter space, in which a large portion of the labels in each cluster are fixed. Then any permutation other than identity gives us bigger mis-clustering rate. 
\begin{proof}[Proof of Theorem \ref{thm:lower_bound}]
For any $z \in [k]^n$, let us define $n_u(z) = \sum_{i=1}^{n} \mathbb{I}\{z_i = u\}$. Let $z^* \in [k]^n$ satisfying $n_1(z^*) \le n_2(z^*) \le \cdots, \le n_k(z^*)$ with $n_1(z^*) = n_2(z^*) = \lfloor \alpha n \rfloor$. It is easy to check the existence of $z^*$. For each $u \in [k]$, we choose a subset of $\{i: z^*(i)=u\}$ with cardinality $\lceil n_u(z^*) - \frac{\alpha n}{4k} \rceil$, denoted by $T_u$. Define $T = \cup_{u=1}^{k} T_u$ and
\begin{equation}
\mathcal{Z}^*  = \left\{ z \in \mathcal{Z}_0, z_i = z_i^* \textrm{ for all } i \in T  \right\}
\end{equation}
A key observation is that for any $z \neq \tilde{z} \in \mathcal{Z}^*$, we have $\frac{1}{n} \sum_{i=1}^{n} \mathbb{I}\{ z_i \neq \tilde{z}_i \} \le \frac{k}{n} \frac{\alpha n}{4k} = \frac{\alpha}{4}$ and
\begin{equation}
\frac{1}{n} \sum_{i=1}^{n} \mathbb{I}\{ \pi(z_i) \neq \tilde{z}_i \} \ge \frac{1}{n} (\alpha n - \frac{\alpha n }{4k}) \ge \frac{\alpha}{2}
\end{equation}
for all $\pi \in \mathcal{S}_k \neq \mathbb{I}_k$. Thus, $\ell(z, \tilde{z}) = \frac{1}{n} \sum_{i=1}^{n} \mathbb{I}\{z_i \neq \tilde{z}_i\}$ for all $z, \tilde{z} \in \mathcal{Z}^*$. Then following the same arguments in the proof of Theorem 2 in \cite{gao2016community}, we can obtain 
\begin{equation}{\label{eq:lower1}}
\inf_{\hat{z}} \sup_{z \in \mathcal{Z}} \E \ell(\hat{z},z) 
\ge \frac{\alpha}{6} \frac{1}{|T^c|} \sum_{i \in T^c} \left[ \frac{1}{2k^2} \inf_{\hat{z}_i} \left(\mathbb{P}_1\{\hat{z}_i = 2\} + \mathbb{P}_2\{\hat{z}_i=1\}\right) \right] 
\end{equation}
Here $\mathbb{P}_t, t \in \{1,2\}$ denote the probability distribution of our data given $z_i = t$. By Neyman-Pearson Lemma, the infimum of the right hand side of (\ref{eq:lower1}) is achieved by the likelihood ratio test 
\begin{eqnarray*}
\hat{z}_i = \argmin_{g \in \{1,2\}} \|y_i - \theta_g\|^2.
\end{eqnarray*}
Thus,
\begin{eqnarray*}
\inf_{\hat{z}_i} \left( \frac{1}{2} \mathbb{P}_1 \left\{ \hat{z}_i = 2 \right\} + \frac{1}{2} \mathbb{P}_2 \left\{ \hat{z}_i = 1 \right\} \right) &=& \mathbb{P} \left\{ \|\theta_1+w_i-\theta_2\|^2 \le  \|w_i\|^2\right\} \\
&=& \mathbb{P} \left\{ \|\theta_1-\theta_2\|^2  \le  2\iprod{w_i}{\theta_1-\theta_2} \right\}.
\end{eqnarray*}
Since $w_{ij}, j \in [d]$ are independent $\mathcal{N}(0,\sigma^2)$,  $\iprod{w_i}{\theta_1-\theta_2} \sim \mathcal{N}(0, \sigma^2 \|\theta_1-\theta_2\|^2)$. Let $\Phi(t)$ be the cumulative function of $\mathcal{N}(0,1)$ random variable. By calculating the derivatives, it can be easily proved that
\[  1 - \Phi(t) = \frac{1}{\sqrt{2\pi}} \int_{t}^{\infty} e^{-x^2/2} dx \ge \frac{1}{\sqrt{2\pi}} \frac{t}{t^2+1} e^{-t^2/2}. \]
Then when $\|\theta_1-\theta_2\| \ge \sigma$, we have
\[ \mathbb{P} \left\{ \iprod{w_i}{\theta_1-\theta_2} \ge \frac{1}{2} \|\theta_1-\theta_2\|^2   \right\} \ge \frac{\sigma}{\sqrt{2 \pi} \|\theta_1-\theta_2\|} \exp \left( - \frac{\|\theta_1-\theta_2\|^2}{8\sigma^2} \right). \]
Consequently,
\begin{eqnarray*}
\inf_{\hat{z}} \sup_{z \in \mathcal{Z}} \E \ell(\hat{z},z)  &\ge& \exp \left( - \frac{\Delta^2}{8\sigma^2} - 2\log \frac{2k \Delta}{\alpha \sigma} \right).
\end{eqnarray*}
The proof is complete.
\end{proof}
\subsection{A counterexample} \label{sec:counter}
Now we give a counterexample to show that the initialization condition in Theorem \ref{thm:kmeans} is almost necessary. Consider a noiseless case with 6 equal size clusters, as showed in Figure \ref{figure:A0}.  Suppose $\|\theta_i-\theta_j\| = \Delta$ for $i \neq j \in [3]$ and $\|\theta_{i+3} - \theta_i\| = \lambda \Delta$ for $i \in [3]$. We are given an initializer that mixes cluster $i$ and cluster $i+3$, $i \in [3]$, with $m/(2\lambda)$ data points from cluster $i+3$ and $m-m/(2\lambda)$ data points from cluster $i$ for some integer $m$. Consequently, the estimated (initialized) center $\hat{\theta}_i^{(0)}$ lies in the middle of two true centers. For the next label update step, since there are ties, we may assign half points at $\theta_1$ to cluster $\hat{\theta}_1^{(0)}$ and another half to cluster $\hat{\theta}_3^{(0)}$ and the estimated centers remain the same as before. Therefore, $\{\hat{\theta}_i^{(0)}, i \in [3]\}$ is a stationary point for Lloyd's algorithm and it may not convergence to the true centers. We would like to note that this counterexample is a worst case in theoretical analysis, which may not happen in practice. 

\begin{figure}[h]
\centering
\includegraphics[width=0.5\textwidth]{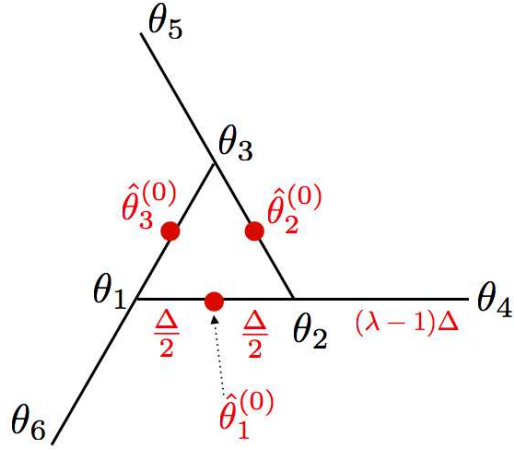}
\caption{A counter example showing that Lloyd's algorithm may not converge when $G_0=\frac{1}{2\lambda}$ or $\Lambda_0 = \frac{1}{2}$.}
\label{figure:A0}
\end{figure}


\subsection{Proofs of Random Initialization}
\begin{proof} [Proof of Theorem \ref{thm:2improved}]
For a data independent initializer, we have the following result. 
\begin{lemma} \label{lm:dataindept}
For any $\delta < 1/4$, if we have a data independent initializer satisfying 
\begin{equation} \label{eq:dataindept}
A_0 \le \frac{1}{2} - \frac{\sigma}{\|\theta^*\|} \sqrt{\frac{2\log(1/\delta)}{n}} \left(1 + \frac{18 \sqrt{d}}{r} \right)
\end{equation}
the conclusion of Theorem \ref{thm:2main} continues to hold with probability greater than $1-\delta$.
\end{lemma}
Given Lemma \ref{lm:dataindept}, let us study the condition under which completely random initialization falls into the basin of attraction (\ref{eq:dataindept}).  When we randomly assign initial labels, we have $|A_0 - 1/2| \ge n^{-1/2}$ with probability greater than 0.3. We can boost this probability to $1-1/\delta$ by independently drawing $3 \log(1/\delta)$ different random initializers. Recall our loss function (\ref{eq:loss}).  We can assume $A_0 < 1/2 - n^{-1/2}$. Otherwise, we may flip all the labels in the initial step. Combining this with condition (\ref{eq:dataindept}),  we have the following corollary. \\

Now it remains to prove Lemma \ref{lm:dataindept}. We focus on the first iteration. By (\ref{eq:decomptheta}), we have $\hat{\theta}^{(1)} - \theta^* = -2A_0 \theta^* +  R$ with $R = \frac{1}{n}\sum_{i=1}^{n}(1-2c_i) w_i$. Then 
\begin{eqnarray*}
\mathbb{I}\{\hat{z}_i^{(1)} \neq z_i\} 
&=& \mathbb{I}\left\{ \iprod{\theta^*+w_i}{(1-2A_0)\theta^*+R} \le 0 \right\} \\ 
&\le& \mathbb{I}\left\{ \iprod{w_i}{\theta^*} \le -\frac{1}{2} \|\theta^*\|^2 \right\} + \mathbb{I}\left\{ \frac{\iprod{R}{\theta^*+w_i}}{1-2A_0} \le - \frac{1}{2}\|\theta^*\|^2 \right\}.
\end{eqnarray*}
Since the initializer is data independent, $\{c_i\} \subseteq \{-1,1\}^{n}$ are independent of $\{w_i\}$. Chernoff's bound implies
\[ \iprod{\theta^*}{R}= \frac{1}{n} \sum_{j=1}^{n} (1-2c_j) w_i' \theta^* \ge - \frac{\sigma \|\theta^*\|\sqrt{2\log(1/\delta)}}{\sqrt{n}} \]
with probability greater than $1-\delta$. Consequently, we obtain
\begin{equation} \label{eq:random1}
A_1 \le \frac{1}{n} \sum_{i=1}^{n} \mathbb{I}\left\{ \iprod{w_i}{\theta^*} \le - \frac{1}{2} \|\theta^*\|^2 \right\} +\frac{1}{n} \sum_{i=1}^{n} \mathbb{I}\left\{ \iprod{R}{w_i} \le - \beta \|\theta^*\|^2 \right\},
\end{equation}
where $\beta = \frac{1}{2}-A_0 - \frac{\sigma \sqrt{2\log(1/\delta)}}{\sqrt{n} \|\theta^*\|}$. By Chernoff's bound and Hoeffding's inequality, the first term in the RHS of (\ref{eq:random1}) can be upper bounded by
\begin{equation*}
\exp \left( - \frac{\|\theta^*\|^2}{8\sigma^2} \right) + \sqrt{\frac{\log(1/\delta)}{n}}
\end{equation*}
with probability greater than $1-\delta$. Markov's inequality and proof of Lemma \ref{lm:tech6} give us
\begin{eqnarray*}
\frac{1}{n} \sum_{i=1}^{n} \mathbb{I}\left\{ \iprod{R}{w_i} \le - \beta \|\theta^*\|^2 \right\} \le \frac{1}{n \beta^2\|\theta^*\|^4} \sum_{i=1}^{n} (R'w_i)^2 \le \frac{9\sigma^2(1+d/n)}{\beta^2\|\theta^*\|^4} \|R\|^2.
\end{eqnarray*}
with probability greater than $1-\exp(-n/2)$. From (\ref{eq:bbb}), we have $\|R\|^2 \le \frac{4\sigma^2 d \log(1/\delta)}{n}$ with probability greater than $1-\delta$. Combining the pieces, we obtain
\begin{equation}
A_1 \le \exp \left( - \frac{\|\theta^*\|^2}{8\sigma^2} \right) + \sqrt{\frac{\log(1/\delta)}{n}} +  \frac{36\sigma^4(1+d/n)}{\beta^2\|\theta^*\|^4}  \frac{d \log(1/\delta)}{n}
\end{equation}
with probability greater than $1-3\delta$. Under the condition that $$ A_0 \le \frac{1}{2} - \frac{\sigma}{\|\theta^*\|} \sqrt{\frac{2\log(1/\delta)}{n}} \left(1 + \frac{18\sigma \sqrt{d(1+d/n)}}{\|\theta^*\|} \right), $$ 
we have $A_1 \le \frac{1}{4}$ with probability greater than $1-3\delta$. Then using the result of Theorem \ref{thm:2main}, the proof is complete.
\end{proof}
\section{Proofs of Community Detection} \label{sec:SBMproof}
\subsection{Proof of Theorem \ref{thm:SBMyu}}
We use the same notation as in the Gaussian case. Let $T_g^*$ be the true cluster $g$ and $T_g^{(s)} $ be the estimated cluster $g$ at iteration $s$. And we will drop the dependency of $s$ when there is no ambiguity in the context. $A_s$ is the mis-clustering rate at iteration $s$ and $G_s$ is the group-wise mis-clustering rate at iteration $s$. Similar to the analysis of Gaussian mixture model, the proof consists of two steps. We first prove that given $G_0 \le \frac{1}{2} - \epsilon_0$ with $\epsilon_0 \ge \frac{C_0\sqrt{a} \beta k \log(\beta k)}{a-b}$ for a sufficiently large constant $C_0$,  we have $G_s \le \frac{1}{3}$ for all $1 \le s \le \lfloor 3\log n \rfloor$ with high probability. Then we prove the mis-clustering rate $A_s$ geometrically decays to the rate (\ref{eq:SBMthm}). \\

\noindent Given $z_i=g$, we decompose (\ref{eq:SBMcenter}) as follows. 
\begin{eqnarray*}
\frac{1}{n_h} \sum_{j=1}^{n} A_{ij} \mathbb{I}\{\hat{z}_j^{(s)}=h\} 
&=& \frac{1}{n_h} \sum_{j=1}^{n} w_{ij} \mathbb{I}\{\hat{z}_j^{(s)}=h\} + \frac{1}{n_h} \sum_{j=1}^{n} \sum_{l=1}^{k} B_{gl} \mathbb{I}\{\hat{z}_j^{(s)}=h, z_j=l\} \\
&=& \frac{1}{n_h} \sum_{j \in T_h} w_{ij} + \frac{a}{n} \frac{n_{gh}}{n_h} + \frac{b}{n} \frac{\sum_{l \neq g} n_{lh}}{n_h} \\
&=& \frac{1}{n_h} W_{i}(T_h) + \frac{a}{n} \frac{n_{gh}}{n_h} + \frac{b}{n} \left(1 - \frac{n_{gh}}{n_h} \right)
\end{eqnarray*}
Consequently, the error of estimating $z_i$ at iteration $s+1$ can be upper bounded as
\begin{eqnarray} 
&& \nonumber \mathbb{I}\{\hat{z}_i^{(s+1)}\neq g, z_i = g\} \\
\nonumber &\le& \mathbb{I}\left\{ \max_{h \neq g} \sum_{j=1}^{n} A_{ij} \mathbb{I}\{\hat{z}_j^{(s)}=h\} \ge \sum_{j=1}^{n} A_{ij} \mathbb{I}\{\hat{z}_j^{(s)}=g\}, z_i=g \right\} \\
\nonumber &\le& \mathbb{I}\left\{ \max_{h \neq g} \frac{1}{n_h} W_{i}(T_h) - \frac{1}{n_g} W_{i}(T_g) \ge \frac{a-b}{n} \left( \frac{n_{gg}}{n_g} - \max_{h \neq g}\frac{n_{gh}}{n_h} \right) \right\} \\
\nonumber &\le& \mathbb{I}\left\{ \max_{h \neq g} \frac{1}{n_h} W_{i}(T_h) - \frac{1}{n_g} W_{i}(T_g) \ge \frac{2\epsilon(a-b)}{n} \right\} \\
\label{eq:SBMerrorz} &\le& \mathbb{I}\left\{ \max_{h \neq g} \frac{1}{n_h} W_{i}(T_h) \ge \frac{\epsilon(a-b)}{n} \right\} +  \mathbb{I} \left\{ \frac{1}{n_g} W_{i}(T_g) \le -\frac{\epsilon(a-b)}{n} \right\}
\end{eqnarray}
where the second inequality is due to our induction assumption that  $n_{gg} \ge (1-1/2+\epsilon)n_g$ and $n_{gh} \le (1/2-\epsilon) n_h$. Union bound implies
\begin{equation} \label{eq:SBMbasic} 
\mathbb{I}\{\hat{z}_i^{(s+1)} \neq z_i\} \le \sum_{h=1}^{k}\mathbb{I}\left\{ W_{i}(T_h) \ge \frac{\epsilon(a-b)n_h}{n} \right\} +  \mathbb{I} \left\{ W_{i}(T_{z_i}) \le -\frac{\epsilon(a-b)n_{z_i}^*}{n} \right\}. 
\end{equation}

Now we give two upper bounds on the sum of indicator variables uniformly over $\{T_h\}$. There proofs are deferred to Section \ref{sec:SBMlemma}.

\begin{lemma} \label{lm:SBMkey} Given $T^* \subseteq [n]$ with cardinality greater than $n/(\beta k)$ and $\epsilon>0$ such that $\epsilon^2 (a-b)^2\ge C_0 \beta a k \log(n/m)$ for a sufficiently large constant  $C_0$. There is an universal constant $C$ such that with probability greater than $1-n^{-2}$, 
\begin{equation} \label{eq:SBMkey}
\sum_{i=1}^{m} \mathbb{I}\left\{\sum_{j \in T}w_{ij} \ge \frac{\epsilon(a-b)|T|}{n} \right\} \le m \exp \left( - \frac{\epsilon^2 (a-b)^2}{C\beta a k} \right) + \frac{|T \Delta T^*|}{5}
\end{equation}
holds for all  $T \subseteq [n]$ with $|T\Delta T^*| \le \frac{1}{2} |T^*|$.
\end{lemma}

\begin{lemma} \label{lm:SBMkey2}Given $T^* \subseteq [n]$ with cardinality greater than $n/(\beta k)$ and $\epsilon>0$ such that $\epsilon^2 (a-b)^2\ge C_0 \beta^2 a k^2 \log(n/m)$ for a sufficiently large constant  $C_0$. There is an universal constant $C$ such that with probability greater than $1-n^{-2}$,  
\begin{equation} \label{eq:SBMkey2}
\sum_{i=1}^{m} \mathbb{I}\left\{ \sum_{j \in T}w_{ij} \ge\frac{\epsilon(a-b)|T|}{n} \right\} \le m \exp \left( - \frac{\epsilon^2 (a-b)^2}{C\beta^2 a k^2} \right) + \frac{|T \Delta T^*|}{4\beta k}
\end{equation}
holds for all  $T \subseteq [n]$ with $|T\Delta T^*| \le \frac{1}{2} |T^*|$.
\end{lemma}

In Lemma \ref{lm:SBMkey} and Lemma \ref{lm:SBMkey2}, we present results for the upper tail $\sum_{i=1}^{m}\mathbb{I}\left\{ W_{i}(T_h) \ge \epsilon(a-b)|T|/n \right\}$. By slightly modifying the proof, the same results hold for the lower tail $\sum_{i=1}^{m} \mathbb{I}\left\{ W_i(T_h) \le - \epsilon(a-b)|T|/n \right\}$. Taking an sum over $i \in T_g^*$  in (\ref{eq:SBMbasic}) and using Lemma \ref{lm:SBMkey2}, we obtain
\begin{eqnarray*}
&& \sum_{i \in T_g^*} \mathbb{I}\{\hat{z}_i^{(s+1)} \neq z_i\} \\
&\le& \sum_{h \neq g} \sum_{i \in T_g^*} \mathbb{I}\left\{ W_{i}(T_h) \ge \frac{\epsilon(a-b)n_h}{k} \right\} + \sum_{i \in T_g^*}  \mathbb{I} \left\{ W_{i}(T_{g}) \le -\frac{\epsilon(a-b)n_h}{k} \right\} \\
&\le& k n_g^* \exp \left( - \frac{\epsilon^2 (a-b)^2}{C\beta^2 a k^2} \right) + \frac{1}{4\beta k} \sum_{h=1}^{k} |T_h \Delta T_h^*| \\
&\le& k n_g^* \exp \left( - \frac{\epsilon^2 (a-b)^2}{C\beta^2 a k^2} \right) + \frac{n}{2\beta k} A_s.
\end{eqnarray*}
with probability greater than $1-k n^{-2}$. Consequently,
\[ \frac{1}{n_g^*} \sum_{h \neq g} n_{gh}^{(s+1)} \le k \exp \left( - \frac{\epsilon^2 (a-b)^2}{C\beta^2 a k^2} \right) + \frac{1}{2} G_s. \]
with probability greater than $1-k n^{-2}$. Here we have used the fact that $A_s \le G_s$. Using similar arguments, we obtain the same high probability upper bound for $\frac{1}{n_h} \sum_{g \neq h} n_{gh}^{(s+1)}$. Therefore, we have
\[ G_{s+1} \le k \exp \left( - \frac{\epsilon^2 (a-b)^2}{C\beta^2 a k^2} \right) + \frac{1}{2} G_s \]
with probability greater than  $1-k^2 n^{-2}$. When $s=0$, we choose $\epsilon=\epsilon_0$. Since $\epsilon_0 \ge \frac{C_0\sqrt{a} \beta k}{a-b}$, we have $G_1 \le 1/12+G_0/2 \le 1/3$ when $C_0 \ge (C+1)\log12$. For $s \ge 1$, we choose $\epsilon = 1/3$ and it is straight forward to prove by induction that $G_{s+1} \le 1/12+1/6 \le 1/3$ for all  $0 \le s \le 3 \log n$, with probability greater than $1- n^{-1}$, provided $n \ge 3k^2 \log n$. \\

\noindent Now we are ready to give the convergence rate of $A_s$. Since $G_s \le 1/3$ for all $s \in [1,3\log n]$,  $(\ref{eq:SBMbasic})$ holds for all $s \in [1,3\log n]$. Taking average over $i \in [n]$ in (\ref{eq:SBMbasic}) and using Lemma \ref{lm:SBMkey} with $\epsilon=1/3$,  we obtain
\begin{eqnarray*}
A_{s+1} 
&\le& \frac{1}{n} \sum_{h =1}^{k} \sum_{i=1}^{n} \mathbb{I}\left\{ W_{i}(T_h) \ge \frac{a-b}{6k} \right\} +  \frac{1}{n} \sum_{g=1}^{k} \sum_{i \in T_g^*} \mathbb{I} \left\{ W_{i}(T_{g}) \le -\frac{a-b}{6k} \right\} \\
&\le& 2 k \exp \left( - \frac{(a-b)^2}{C\beta a k^2} \right) + \frac{2}{5n} \sum_{h=1}^{k} |T_h \Delta T_h^*| \\
&\le& \exp \left( - \frac{(a-b)^2}{2C\beta a k^2} \right) + \frac{4}{5} A_s
\end{eqnarray*}
with probability greater than $1-k n^{-2}$, where $C$ is some universal constant.

\subsection{Proof of Lemma \ref{lm:SBMkey} and Lemma \ref{lm:SBMkey2}} \label{sec:SBMlemma}
For any fixed $T$ with $|T \Delta T^*| \le \gamma n/k$, let 
$$a_i(T) = \mathbb{I}\left\{  \sum_{j \in T} w_{ij} \ge \frac{\epsilon(a-b)|T|}{n}\right\}.$$ By Bernstein's inequality, the success probability of $a_i(T)$ is upper bounded by
\begin{eqnarray*}
\mathbb{P} \left\{ \sum_{j \in T} w_{ij} \ge  \frac{\epsilon(a-b)|T|}{n} \right\} 
&\le& \exp \left( - C_1 \min\left\{\frac{\epsilon(a-b)|T|}{n}, \frac{\epsilon^2 (a-b)^2 |T|}{a n}  \right\} \right) \\
&\le& \exp \left( - \frac{C_2 \epsilon^2 (a-b)^2}{\beta a k} \right),
\end{eqnarray*}
for some universal constant $C_1$ and $C_2$. Here the last inequality is due to the fact that $|T| \ge |T^*| - |T \Delta T^*| \ge n/(2\beta k)$.  Note that $\{a_i(T), i \in [n]\}$ are independent Bernoulli random variables.
By Bennett's inequality (see Lemma 9 in \cite{gao2015achieving}), for any $t>0$,
\begin{eqnarray} 
\nonumber \mathbb{P} \left\{ \sum_{i=1}^{m} a_i (T) \ge m p + t \right\} 
&\le& \exp \left( t-(mp + t) \log \left( 1+\frac{t}{mp} \right) \right) \\
\label{eq:SBMbennett} &\le& \exp \left( - t \log \left( \frac{t}{emp} \right) \right) ,
\end{eqnarray} 
where $$p = \frac{1}{n}\sum_{i=1}^{n} \E a_i(T) \le \exp \left( - \frac{C_2\epsilon^2 (a-b)^2}{\beta a k} \right)\triangleq\exp \left( - R \right).$$ To prove (\ref{eq:SBMkey}) holds for all $T \subseteq [n]$ with  $|T \Delta T^*| \le \frac{\gamma n}{k}$, we need a chaining argument. For $s=m\exp(-R/30)$, define $\mathcal{D}_0 = \left\{ T \subseteq [n],  |T \Delta T^*| \le s \right\}$ and 
$$\mathcal{D}_r = \left\{T \subseteq [n],  2^{r-1} s \le |T \Delta T^*| \le 2^{r} s \right\}.$$ Then $\mathcal{D} \subseteq \cup_{r=0}^{u} \mathcal{D}_r$ with $u = \lceil \log n \rceil$. Union bound implies\begin{eqnarray*}
\mathcal{A} &\triangleq& \mathbb{P} \left\{ \exists ~T \in \mathcal{D}, ~s.t. ~\sum_{i=1}^{m} a_i(T) \ge m p + \frac{1}{5} |T \Delta T^*| + s \right\} \\
&\le& \sum_{r=0}^{u} \mathbb{P} \left\{ \exists ~T \in \mathcal{D}_r, ~s.t. ~\sum_{i=1}^{m} a_i(T) \ge mp+\frac{1}{5} |T \Delta T^*| + s \right\} \\
&\le& \sum_{r=0}^{u} |\mathcal{D}_r| \max_{T \in \mathcal{D}_r} \mathbb{P} \left\{ \sum_{i=1}^{m} a_i(T) \ge mp+\frac{1}{5} |T \Delta T^*| + s \right\} \\
\end{eqnarray*}
Note that $T = T^* \setminus (T^* \cap T^c) \bigcup (T \cap (T^*)^c)$. Given $T^*$ and $|T \Delta T^*|=u$, there are at most ${n \choose u} 2^u $ possible choices of $T$. Thus, the cardinality of $\mathcal{D}_r$ is upper bounded by 
\begin{equation*}
\sum_{u=0}^{2^{r}s} {n \choose u} 2^u \le \sum_{u=0}^{2^{r}s} \left(\frac{2en}{u} \right)^u \le 2^r s \exp\left(2^{r} s \log \left( \frac{2e n}{2^{r} s} \right) \right) \le \exp\left(2^{r} s \log \left( \frac{e n}{2^{r-2} s} \right) \right).
\end{equation*} 
Combining this with (\ref{eq:SBMbennett}), we obtain
\begin{eqnarray*}
&& |\mathcal{D}_0| \max_{T \in \mathcal{D}_0} \mathbb{P} \left\{ \sum_{i=1}^{m} a_i(T) \ge mp+\frac{1}{5} |T \Delta T^*| + s \right\} \\
&\le& \exp \left( - s \log \left( \frac{s}{emp} \right) + s \log \left( \frac{2en}{s} \right)\right) \\
&\le& \exp \left( - \frac{Rs}{2} \right)
\end{eqnarray*}
and
\begin{eqnarray*}
&& |\mathcal{D}_r| \max_{T \in \mathcal{D}_r} \mathbb{P} \left\{ \sum_{i=1}^{m} a_i(T) \ge mp+\frac{1}{5} |T \Delta T^*| + s \right\} \\
&\le& \exp \left( - \frac{2^{r}}{10} s \log \left( \frac{2^{r}s}{10emp} \right) + 2^{r} s \log \left( \frac{en}{2^{r-2}s} \right)\right) \\
&\le& \exp \left( - \frac{2^{r} Rs}{20} \right)
\end{eqnarray*}
for all $r \ge 1$, where we have used the assumption that $\frac{(a-b)^2}{\beta a k} \ge C\log \frac{n}{m}$ for a sufficiently large constant $C$. Thus,
\begin{eqnarray*}
\mathcal{A} \le \exp \left( - \frac{Rs}{2} \right) + \sum_{r=1}^{u} \exp\left( - \frac{2^{r} Rs}{20} \right) \le 3\exp \left( - \frac{Rs}{10} \right).
\end{eqnarray*}
By our assumption, we have $R = \frac{C_2 \epsilon^2 (a-b)^2}{\beta a k } \ge \frac{3}{4} R + 30\log \frac{n}{m}$. Consequently, $\mathcal{A} \le 3\exp \left( - 3 s \log \frac{n}{s} \right) \le n^{-2}$. Therefore, the proof Lemma \ref{lm:SBMkey} is complete. Lemma \ref{lm:SBMkey2} follows from almost identical arguments except that we choose $s = m\exp\left(-\frac{R}{8 \beta k}\right)$.

\section{Proofs of Crowdsourcing}
\subsection{Proof of Theorem \ref{thm:CSmvoting}} \label{sec:MVproof}
Suppose $z_j = g$, then Chernoff bound gives us
\begin{eqnarray*}
\mathbb{P}\left\{\hat{z}_j^{(0)} \neq z_j \right\} 
&=& \mathbb{P}\left\{ \exists h \neq g, \sum_{i=1}^{m} \mathbb{I}\{X_{ij}=h\} \ge \sum_{i=1}^{m} \mathbb{I}\{X_{ij}=g\} \right\} \\
&\le& \sum_{h \neq g}  \mathbb{P}\left\{\sum_{i=1}^{m} \left( \mathbb{I}\{X_{ij}=h\} - \mathbb{I}\{X_{ij}=g\} \right)>0 \right\} \\
&\le& \sum_{h \neq g} \prod_{i=1}^{m} \E \exp \left( \lambda \mathbb{I}\{X_{ij}=h\} -  \lambda \mathbb{I}\{X_{ij}=g\} \right).
\end{eqnarray*}
for all $\lambda >0 $. Since $X_{ij}$ is from a multinomial distribution $(\pi_{ig1}, \pi_{ig2}, \cdots, \pi_{igk})$, we have
\begin{eqnarray*}
&&\prod_{i=1}^{m} \E \exp \left( \lambda \mathbb{I}\{X_{ij}=h\} -  \lambda \mathbb{I}\{X_{ij}=g\} \right) \\
&=& \prod_{i=1}^{m} \left( \pi_{igh} e^{\lambda} + \pi_{igg} e^{-\lambda} + (1 - \pi_{igh} - \pi_{igg}) \right) \\
&\stackrel{\text{(a)}}{\le}& \prod_{i=1}^{m} \exp \left( \pi_{igh} e^{\lambda} + \pi_{igg} e^{-\lambda}-  \pi_{igh} -  \pi_{igg} \right)\\\
&=& \exp \left( Ae^{\lambda} + Be^{-\lambda} - A - B  \right),
\end{eqnarray*}
where $A = \sum_{i=1}^{m}\pi_{igh}$ and $B = \sum_{i=1}^{m} \pi_{igg}$. Here inequality (a) is due to the fact that $1+x \le e^x$ for all $x$. Choosing $\lambda = 0.5\log (B/A)$ yields
\begin{eqnarray*}
\mathbb{P}\left\{\hat{z}_j^{(0)} \neq z_j \right\}  &\le& \sum_{h \neq g} \exp \left( - \left( \sqrt{\sum_{i=1}^{m} \pi_{igg}} - \sqrt{\sum_{i=1}^{m} \pi_{igh}} \right)^2 \right) \\
&\le& k \exp \left( - V(\pi) \right)
\end{eqnarray*}
where $V(\pi)=\min_{h \neq g} \left( \sqrt{\sum_{i=1}^{m} \pi_{igg}} - \sqrt{\sum_{i=1}^{m} \pi_{igh}} \right)^2$. Then using Markov's inequality, we have
\begin{eqnarray*}
\mathbb{P} \left\{ \frac{1}{n} \sum_{j=1}^{n} \mathbb{I}\{\hat{z}_j^{(0)} \neq z_j \} \ge t \right\} \le \exp \left( - V(\pi) + \log k  - \log t \right)
\end{eqnarray*}
for all $t>0$. A choice of $t=1/(4k)$ completes the proof.

\subsection{Proof of Corollary \ref{thm:crowdsourcing}} \label{sec:CSproof}
Let us first calculate the sub-Gaussian parameter of $w_i$. For any $\lambda>0$ and $a=(a_{11}, a_{12}, \cdots, a_{mk}) \in \mathbb{R}^{mk}$ with $\|a\|=1$,
\[ \E \exp \left( \lambda \iprod{a}{w_i} \right) = \prod_{u=1}^{m} \E \exp \left( \lambda \sum_{h=1}^{k} w_{uih} a_{uh} \right). \]
For any fixed $u$, define a Multinoulli  random variable $X$ such that \[ \mathbb{P}\{X = a_{uh}\} = \pi_{u z_i h}, h \in [k]. \]
Recall the definition of $w_{uih}$ that $w_{uih}=\mathbb{I}\{X_{ui=h}\} - \pi_{u z_i h}$. Then,
\[ \E \exp \left( \lambda \sum_{h=1}^{k} w_{uih} a_{uh} \right) =  \E e^{\lambda(X-\E X)} \le e^{2\lambda^2 \max_{h \in [k]} a_{uh}^2}, \]
where the last inequality is because $|X| \le \max_{h \in [k]} |a_{uh}|$ is a bounded random variable. Consequently, we have
\[ \E \exp \left( \lambda \iprod{a}{w_i} \right) \le \E \exp \left( 2\lambda^2 \sum_{u=1}^{m} \max_{h \in [k]} a_{uh}^2 \right) \le \E \exp \left( 2\lambda^2 \right) . \]
Therefore, the sub-Gaussian parameter of $w_i$ is upper bounded by 2.

\section{Proofs of Technical Lemmas} \label{sec:techproof}
To help readers better understand the results, we first give proofs for spherical Gaussians and then extend it to the general sub-Gaussians. Without loss of generality, we assume $\sigma=1$ in this section. 

\subsection{Proofs for spherical Gaussians}
We denote by $\mathcal{N}(0,\Sigma)$ the normal random vectors with mean zero and covariance matrix $\Sigma$. For $Z \sim \mathcal{N}(0,1)$, we have the following two facts on its moment generating functions. 
\begin{equation} \label{eq:Nfact1}
\E e^{\lambda Z_1} = e^{\frac{\lambda^2}{2}}.
\end{equation}
\begin{equation} \label{eq:Nfact2}
\E e^{\lambda Z_1^2} = \frac{1}{\sqrt{1-2\lambda}} \quad \textrm{ for all } \lambda <1/2.
\end{equation}
\begin{proof}[Proof of Lemma \ref{lm:tech4}]
For any fixed $S \subseteq [n]$, $W_S = \sum_{i \in S} w_i \sim \mathcal{N}(0, |S| I_d)$. By (\ref{eq:Nfact2}), we have
$$\E e^{ \lambda \|W_S\|^2} 
= (1-2\lambda|S|)^{-d/2}$$
for $\lambda < 1/(2|S|)$. Then Chernoff's bound yields
\begin{equation}\label{eq:bbb}
\mathbb{P}\left\{ \|W_S\|^2 \ge t \right\}  \le \exp \left( - \lambda t - \frac{d}{2} \log (1-2\lambda|S|) \right). 
\end{equation}
Choosing $\lambda=0.49/|S|$ and $t=1.62(n+4d)|S|$, 
$$ \mathbb{P}\left\{ \|W_S\|^2 \ge 1.62(n+4d)|S| \right\} \le \exp \left( - 0.7938n - 3.17d + \frac{d}{2} \log 50 \right).  $$
Since there are $2^n$ subsets of $[n]$, an union bound argument gives us
\begin{eqnarray*}
\mathbb{P} \left\{ \max_{S}\left(\| W_S \| - \sqrt{1.62 (n+4d)|S|}\right) \ge 0 \right \}
\le 2^n \exp\left(- 0.7938n \right)
\le \exp(- 0.1n).
\end{eqnarray*}
\end{proof}

\begin{proof}[Proof of Lemma \ref{lm:tech1}]
Since $w_i' \theta^* \stackrel{i.i.d} \sim \mathcal{N}(0,\|\theta^*\|^2)$, we have $\bar{w}'\theta^* \sim \mathcal{N}(0,\|\theta^*\|^2/n)$. Consequently, 
\[ \mathbb{P} \left\{ - \bar{w}'\theta^* \ge \frac{\|\theta^*\|^2}{\sqrt{n}} \right\} \le \exp \left( - \frac{\|\theta^*\|^2}{2} \right). \]
Choosing $\lambda=0.4/n$ in (\ref{eq:bbb}), we have
\[ \mathbb{P} \left\{ \|\bar{w}\|^2 \ge t \right \} \le \exp\left(- 0.4nt + \frac{d}{2} \log 5 \right). \]
A choice of $t = \frac{1}{n} \left( 3d \sigma^2 + \|\theta^*\|^2\right)$ yields the desired result. 
\end{proof}

\begin{proof}[Proof of Lemma \ref{lm:tech6}]
Let $A=[w_1, w_2, \cdots, w_n]$. Then 
$$\sup_{\|a\|=1} \sum_{i=1}^{n} (a'w_i)^2 = \sup_{\|a\|=1} \|Aa\|_2^2.$$ Let $\mathcal{B}_1$ be the unit ball in $\mathbb{R}^d$ and $\mathcal{C}$ be a $\epsilon$-net of $\mathcal{B}_1$ such that for any $a \in \mathcal{B}_1$, there is a $b \in \mathcal{C}$ satisfying $\|a-b\| \le \epsilon$.  Then we have $|\mathcal{C}| \le (1+2/\epsilon)^d$ \cite[Lemma 4.1]{pollard90} and 
\begin{eqnarray*}
\|Aa\| &\le& \|Ab\| + \|A(a-b)\| \\
&\le& \max_{b \in \mathcal{C}} \|Ab\| + \|a-b\| \sup_{a \in \mathcal{B}_1} \|Aa\| \\
&\le& \max_{b \in \mathcal{C}} \|Ab\| + \epsilon \sup_{a \in \mathcal{B}_1} \|Aa\|.
\end{eqnarray*}
Taking a supreme over $a \in \mathcal{B}_1$ on both sides and rearranging, we get 
\begin{equation} \label{eq:covering}
\sup_{a \in \mathcal{B}_1} \|Aa\| \le \frac{1}{1-\epsilon}\max_{b \in \mathcal{C}} \|Ab\|.
\end{equation}
For any fixed $b \in \mathcal{B}_1$, $b'w_i \stackrel{i.i.d} \sim \mathcal{N}(0,1)$.  Then by (\ref{eq:Nfact2}), 
\[ \mathbb{P} \left\{ \sum_{i=1}^{n} (b'w_i)^2 \ge t \right\} \le e^{- 0.25t} \prod_{i=1}^{n} \E e^{ 0.25(b'w_i)^2} \le \exp \left( - 0.25t + 0.35 n \right), \]
where we have set $\lambda=0.25$. Then (\ref{eq:covering}) and union bound give us
\begin{eqnarray*} 
\mathbb{P} \left\{ \sup_{a \in \mathcal{B}_1} \|Aa\|^2 \ge t \right\} &\le& \sum_{b \in \mathcal{C}} \mathbb{P} \left\{ \|Ab\|^2 \ge (1-\epsilon)^2 t \right\} \\
&\le& \exp \left( - 0.25(1-\epsilon)^2t+ 0.35 n + d \log (1+2/\epsilon) \right). 
\end{eqnarray*}
A choice of $\epsilon=0.05$ and $t=2(n+9d)$ yields the desired result. 
\end{proof}

\begin{proof}[Proof of Lemma \ref{lm:tech5}]
For each $i \in [n]$, define $G_i=S \cap \{i\}^c$. Then,
\begin{equation} \label{eq:tech5term} 
\iprod{w_i}{\sum_{l \in S} w_l} \le \iprod{w_i}{\sum_{l \in G_i} w_l} + \|w_i\|^2
\end{equation}
$W_{G_i} = \sum_{l \in G_i} w_l$ is an isotopic Gaussian random vector with variance smaller than $|S|$ on every direction. By (\ref{eq:Nfact1}) and (\ref{eq:Nfact2}), for any independent $Z_1 \sim N(0,1)$ and $Z_2 \sim N(0,1)$,
\[ \E e^{\lambda Z_1 Z_2} = \E e^{\frac{\lambda^2 Z_1^2}{2}} = \frac{1}{\sqrt{1-\lambda^2}}  \]
for all $\lambda \in [0,1)$.
Since $w_{i}$ and $W_{G_i}$ are independent, Chernoff bound implies
\begin{equation*} \label{eq:tech5term}
\mathbb{P}\left\{ \iprod{w_i}{W_{G_i}}  \ge \sqrt{|S|} t \right\}  \le  \exp \left( - \lambda t - \frac{d}{2} \log(1-\lambda^2) \right) \le \exp \left( - \lambda t + \lambda^2 d \right).
\end{equation*}
for $\lambda \le 1/2$. Choosing $\lambda=\min\{\frac{t}{2d},\frac{1}{2}\}$ yields
\begin{eqnarray} \label{eq:tech5term1}
\mathbb{P}\left\{ \iprod{w_i}{W_{G_i}}  \ge \sqrt{|S|} t \right\}  \le  \exp \left( - \min \left\{\frac{t^2}{4d}, \frac{t}{4}\right\} \right).
\end{eqnarray}
From (\ref{eq:bbb}), a choice of $\lambda=1/3$ gives us
\[ \mathbb{P} \left\{ \|w_i\|_2^2 \ge t \right\} \le \exp\left( - \frac{t}{3} + 0.55 d\right). \] 
Setting $t=3(d+4\log n)$, we obtain
\begin{equation} \label{eq:tech5term2}
\mathbb{P} \left\{ \|w_i\|_2^2 \ge 3\sigma^2(d+\log(1/\delta)) \right\} \le \delta
\end{equation}
with probability greater than $1-\delta$. Combining (\ref{eq:tech5term}), (\ref{eq:tech5term1}) and (\ref{eq:tech5term2}) completes the proof.
\end{proof}

\subsection{Proofs of sub-Gaussians}
\begin{proof}[Proof of Lemma \ref{lm:tech1.1}]
For any fixed $S \subseteq [n]$, $W_S$ is a $d$ dimensional random vector satisfying
$$ \E e^{\iprod{a}{W_S}} = \prod_{i \in S} \E e^{ \iprod{a}{w_i}} \le e^{\frac{|S| \|a\|^2}{2}} $$
for all $a \in \mathbb{R}^d$. By \cite[Theorem 2.1]{hsu2012tail}, we have
\begin{equation} \label{eq:techkey}
\mathbb{P}\left\{ \|W_S\|^2 \le |S|(d+2\sqrt{dt}+2t) \right\} \le \exp(-t).
\end{equation}
Note that there are at most $2^n$ possible choices of $S$. Using the union bound completes the proof.
\end{proof}

\begin{proof}[Proof of Lemma \ref{lm:tech1.2}]
 For any fixed $b \in \mathbb{R}^{d}$ with $\|b\|=1$, $b'w_i's$ are independent sub-gaussian random variables with parameter $1$. Using 
\cite[Theorem 2.1]{hsu2012tail} again, 
\begin{equation}
\mathbb{P}\left\{ \sum_{i=1}^{n} (b'w_i)^2 \le n+2\sqrt{nt}+2t \right\} \le \exp(-t).
\end{equation}
Following the same argument as in the proof of Lemma \ref{lm:tech6} and choosing $t=2d+n$, we obtain the desired result. 
\end{proof}

\begin{proof}[Proof of Lemma \ref{lm:tech1.3}]
Similarly to the proof of Lemma \ref{lm:tech5}, we upper bound $\iprod{w_i}{W_S}$ as 
\[ \iprod{w_i}{W_S} \le \iprod{w_i}{W_{G_i}} + \|w_i\|^2.  \]
Since $w_i$ and $W_{G_i}$ are independent sub-Gaussian random vectors, we have
\[ \E e^{\lambda \iprod{w_i}{W_{G_i}}} \le \E e^{\frac{\lambda^2\|W_{G_i}\|^2}{2}} \le \E e^{\frac{\lambda^2 |S| \|z\|^2}{2}} \]
for $\lambda^2 \le 1/|S|$, where $z \sim \mathcal{N}(0, I_d)$. Here the last inequality is due to \cite[Remark 2.3]{hsu2012tail}. Then \ref{eq:tech5term1}) also holds for sub-Gaussian random vectors. Applying \cite[Theorem 2.1]{hsu2012tail} on $w_i$, we have
\[ \mathbb{P}\left\{ \|w_i\|^2 \ge d+2\sqrt{d\log(1/\delta)} + 2\log(1/\delta)\right\} \le \delta \]
for $\delta>0$. Following the same arguments as in the proof of Lemma \ref{lm:tech5}, the proof is complete.
\end{proof}

\begin{proof}[Proof of Lemma \ref{lm:tech3}]
For any $h \in [k]$, (\ref{eq:techkey}) implies
\[ \mathbb{P}\left\{ \|W_{T_h^*}\|^2 \le |T_h^*|(d+4\sqrt{d\log n}+8\log n) \right\} \le n^{-4}.\]
Since $n \ge k$, a union bound argument completes the proof.
\end{proof}

\begin{proof}[Proof of Lemma \ref{lm:tech2}]
Let $u_i = \mathbb{I}\left\{ a\|\theta_h-\theta_g\|^2 \le \iprod{w_i}{\|\theta_h-\theta_g\|} \right\}$. Then $u_i, i \in T_g^*$ are independent Bernoulli random variables. Hoeffding's inequality implies
\[ \sum_{i \in T_g^*} u_i \le \sum_{i \in T_g^*} \E u_i  + 2 \sqrt{ n_g^* \log n}  \]
with probability greater than $1-n^{-4}$. Using Chernoff bound and the sub-Gaussian property of $w_i$, we have
\begin{eqnarray*} 
\E u_{i} &\le& \E \exp \left( - \lambda a \|\theta_h-\theta_g\|^2 + \lambda \iprod{w_i}{\|\theta_h-\theta_g\|} \right) \\
&\le& \exp \left( - \lambda a \|\theta_h-\theta_g\|^2 + \lambda^2  \|\theta_h-\theta_g\|^2/2 \right) \\
&=& \exp \left( - a^2 \|\theta_h-\theta_g\|^2/2 \right),
\end{eqnarray*}
where we set $\lambda=a$ in the last equality. By union bound and the fact that $n \ge k$, we get the desired result.  
\end{proof}



\bibliography{mixture}
\bibliographystyle{plain}

\end{document}